\documentclass[12pt]{amsart}
\usepackage{amsmath}
\usepackage{amsfonts,amssymb,amsthm, txfonts, pxfonts,amscd}
\usepackage[dvips]{graphicx}
\usepackage{epsfig,subfigure}
\usepackage{geometry,psfrag}

\def\struckint{\mathop{%
\def\mathpalette##1##2{\mathchoice{##1\displaystyle##2}%
  {##1\textstyle##2}{##1\scriptstyle##2}{##1\scriptscriptstyle##2}}%
\mathpalette
{\vbox\bgroup\baselineskip0pt\lineskiplimit-1000pt\lineskip-1000pt
\halign\bgroup\hfill$}
{##$\hfill\cr{\intop}\cr\diagup\cr\egroup\egroup}%
}\limits}

\newcommand{\R}{\mathbb{R}}

\newcommand{\C}{\mathbb{C}}
\newcommand{\Q}{\mathbb{Q}}
\newcommand{\Z}{\mathbb{Z}}

\newcommand{\SU}{\mathbb{S}}

\newcommand{\QQQ}{\mathcal{Q}}

\newcommand{\MMM}{\mathcal{M}}

\newcommand{\OOO}{\mathcal{O}}

\newcommand{\HHH}{\mathcal{H}}

\newcommand{\LLL}{\mathcal{L}}

\newcommand{\SL}{{\rm SL}}
\newcommand{\GL}{{\rm GL}}
\newcommand{\Gal}{{\rm Gal}}

\newcommand{\ol}{\overline}
\newcommand{\tr}{{\rm tr}}
\newcommand{\fix}{{\rm fix}}
\newcommand{\odd}{{\rm odd}}

\newcommand\Sl{\textrm{SL}_2(\R)}

\newcommand{\Aff}{{\rm Aff}}

\newtheorem{Theorem}{Theorem}[section]
\newtheorem{Corollary}[Theorem]{Corollary}
\newtheorem{Proposition}[Theorem]{Proposition}
\newtheorem{Lemma}[Theorem]{Lemma}
\newtheorem{Claim}[Theorem]{Claim}

\theoremstyle{remark}
\newtheorem{Remark}{Remark}[Theorem]

\theoremstyle{definition}
\newtheorem*{Acknowledgments}{Acknowledgments}
\newtheorem*{Reader's guide}{Reader's guide}
\newtheorem{Definition}[Theorem]{Definition}

\begin{document}
\title[Teichm\"uller discs in Q(1,1,1,1)]
{Completely periodic directions and orbit closures of
many pseudo-Anosov Teichmueller discs in Q(1,1,1,1)}

\author{Pascal Hubert, Erwan Lanneau, Martin M\"oller}
\address{
Laboratoire d'Analyse, Topologie et Probabilit\'es (LATP) \newline
Case cour A Facult\'e de Saint J\'er\^ome Avenue Escadrille Normandie-Niemen \newline
13397, Marseille cedex 20, France
}

\email{hubert@cmi.univ-mrs.fr}

\address{
Centre de Physique Th\'eorique (CPT), UMR CNRS 6207 \newline
Universit\'e du Sud Toulon-Var and \newline
F\'ed\'eration de Recherches des Unit\'es de 
Math\'ematiques de Marseille \newline
Luminy, Case 907, F-13288 Marseille Cedex 9, France
}

\email{lanneau@cpt.univ-mrs.fr}

\address{
Max-Planck-Institut f\"ur Mathematik\newline
Postfach 7280\newline
53072 Bonn, Germany
}

\email{moeller@mpim-bonn.mpg.de}

\date{\today}
\begin{abstract}
In this paper, we investigate the closure of a large class of
Teichm\"uller discs in the stratum $\QQQ(1,1,1,1)$ or equivalently, 
in a $\GL^+_2(\R)$-invariant locus $\LLL$ of translation surfaces of
genus three. We describe a systematic way to prove that the 
$\GL^+_2(\R)$-orbit closure of a translation surface in $\LLL$ is the 
whole of $\LLL$. The strategy of the proof is an analysis of 
completely periodic directions on such a surface and an
iterated application of Ratner's theorem to unipotent subgroups acting
on an ``adequate'' splitting.
\par
This analysis applies for example to all Teichm\"uller discs stabilized 
obtained by Thurston's construction with a trace field of degree
three which moreover ``obviously not Veech''.
\par
We produce an infinite series of such examples and show moreover
that the favourable splitting situation does not arise
everywhere on $\LLL$, contrary to the situation in genus two.
\par
We also study completely periodic directions on translation surfaces
in $\LLL$. For instance, we prove that completely periodic directions
are dense on surfaces obtained by Thurston's construction.
\end{abstract}
\setcounter{tocdepth}{1}

\maketitle
\tableofcontents
%

\section*{Introduction}
For translation surfaces in genus two, $\GL_2^+(\R)$-orbit 
closures, completely periodic translation surfaces and many
more properties have been classified by Calta, McMullen and e.g.\ 
\cite{EMS}, \cite{HuLe}.
For half-translation surfaces, i.e.\ for pairs $(X,q)$ of a Riemann
surface and a quadratic differential, as well as for surfaces of genus
$g \geq 3$ the situation is much more complicated. At present, all
classification questions are open. In~\cite{HLM} and in the present paper,
we study translation surfaces belonging to $\LLL$, the hyperelliptic locus of the 
non hyperelliptic connected component of the stratum $\HHH(2,2)$ (the 
moduli space of Abelian differentials having two double zeroes). The locus
$\LLL$ is closed and $\GL_2^+(\R)$-invariant. It is more natural 
to study $\LLL$ than it sounds on a first reading since it is 
$\GL_2^+(\R)$-equivariantly isomorphic to $\QQQ(1,1,1,1)$, the principal
stratum of quadratic differentials in genus $2$. We are thus studying
the next simplest cases besides translation surfaces in genus two.

Projections of $\GL_2^+(\R)$-orbits of translation surfaces
to Teichm\"uller space give rise to 
Teichm\"uller discs. The setwise stabilizer of a Teichm\"uller disc 
in the mapping class group is a subgroup of $\SL_2(\R)$, that
reflects the geometry of the original translation surface. See
Section~\ref{background} for details on the basic notions and
references.
\par
We focus in this paper on Teichm\"uller discs that are stabilized by a
pseudo-Anosov diffeomorphism (pseudo-Anosov Teichm\"uller discs
or pseudo-Anosov translation surfaces, for short). The Arnoux-Yoccoz
example (\cite{Arnoux:Yoccoz}, \cite{Ar2}) is a pseudo-Anosov
Teichm\"uller disc with many very exotic properties.
In \cite{HLM}, we proved that the $\GL_2^+(\R)$-orbit closure of the
Arnoux-Yoccoz Teichm\"uller disc is the whole locus
$\LLL$. Rephrasing the previous result in the language of quadratic
differentials, we proved the existence of  
a pseudo-Anosov Teichm\"uller disc with a dense $\GL_2^+(\R)$-orbit in 
$\QQQ(1,1,1,1)$.  
This is very different from the behavior described by McMullen for
Abelian differentials in genus $2$:  every pseudo-Anosov disc is
contained in the eigenform locus over a Hilbert modular surface.  

In this paper,  we describe a systematic way to prove that 
the $\GL_2^+(\R)$-orbit
closure of a surface in $\LLL$ is the whole locus.  This analysis
applies to a large family of pseudo-Anosov Teichm\"uller discs obtained by
the Thurston (or Thurston-Veech) construction 
(see \cite{Thurston,Vee89}). Recall that a direction is 
completely periodic on a translation surface if the surface is
decomposed into a union of cylinders in this direction.  We have 

\begin{Theorem}
\label{closure} 
Let $(X, \omega)\in\LLL$ be a surface given by Thurston's construction with
trace field of degree $3$. Let us assume that $(X,\omega)$ is not a
Veech surface for the most obvious reason: there exists a completely
periodic direction which is not parabolic. Then
$$
\ol{\GL^+_2(\R)\cdot(X,\omega)} = \LLL.
$$
\end{Theorem}

Without the trace field condition the statement is false. Counterexamples
are given by translation surfaces that arise as coverings from genus
one or genus two.
\par
Of course, by the ergodicity of the geodesic flow on any stratum, the 
$\GL^+_2(\R)$-orbit closure of a generic surface in $\LLL$ equals $\LLL$.
While pseudo-Anosov Teichm\"uller discs as in the theorem above behave 
like generic surfaces, they have quite remarkable  topological properties.
Without restrictions on the trace field we show:

\begin{Theorem}
\label{complete-periodicity} 
Let $(X, \omega)\in\LLL$ be a surface stabilized by a pseudo-Anosov
diffeomorphism. Let us assume that there exists a completely periodic direction. 
Then the set of completely periodic directions on $(X, \omega)$
is a dense subset of the circle $\SU^1$.
\end{Theorem}

Obviously, the conclusion fails for a generic surface in $\LLL$, see Proposition~\ref{noncpgeneric}.
Note that the last result applies to translation surfaces given by 
Thurston's construction. Combining these last two theorems, one
gets:

\begin{Theorem}
Let $(X, \omega)\in\LLL$ be a surface given by Thurston's construction with
trace field of degree $3$. Then at least one of the following holds:
\begin{enumerate}
\item The closure of the orbit $\GL^+_2(\R)\cdot(X,\omega)$
  is the whole locus $\LLL$.

\item The limit set of the Veech group $\SL(X,\omega)$ is the full circle.
\end{enumerate}
\end{Theorem}

\begin{Remark}
We do not know if there exists a surface which satisfies these two
properties at the same time.
\end{Remark}

Our results apply to an infinite family of surfaces, for instance a series of
surfaces arising from Thurston's construction. \medskip

\begin{Reader's guide}

We end this introduction by explaining the organization of the paper and by
sketching a proof of the main results.
\par
The strategy of the proof of Theorem~\ref{closure} is to use an
``adequate irrational'' splitting on a surface $(Y,\eta)$ belonging to the closure of the
$\Sl$-orbit of $(X,\omega)$. Such a splitting is given by four
homologous saddle connections which decompose the surface into two
fixed tori and two exchanged cylinders. We will call such a
decomposition a $2T_{fix}2C$ splitting. Following the strategy presented
in~\cite{Mc3} and~\cite{HLM}, we conclude by applying Ratner's theorem~\cite{Ra} to
a cyclic unipotent subgroup acting on the product of the space of pairs of area
$1$ lattices and the space of cylinders.
\par
We use the non parabolic completely periodic direction to obtain the
surface $(Y,\eta)$ by the following way. In Section~\ref{configs} we describe all completely
periodic configurations of surfaces in $\LLL$. For each of these
configurations, in Section~\ref{reductionsteps}, we apply Ratner's theorem to
$\left(\Sl/\textrm{SL}_2(\Z)\right)^3$ the space of triples of normalized lattices
in order to get a resplitting of $(X,\omega)$ into $(Y,\eta)$
with an irrational $2T_{fix}2C$. This can only work under some mild
irrationality hypothesis, which will be a consequence of the trace
field condition and the existence of a non-parabolic direction.
\par
We now give a sketch of the proof of
Theorem~\ref{complete-periodicity}. Following Calta-Smillie
\cite{CaSm}, the hypothesis implies that the Sah-Arnoux-Fathi
(SAF)-invariant \cite{Ar1} is zero for every direction of the holonomy field. Given a
direction $\theta$ containing a cylinder fixed by the hyperelliptic
involution, we remove this cylinder from the surface and  obtain a
genus $2$ translation surface $Y$ with boundary. Using the fact that
the SAF-invariant vanishes on $Y$ and the list of periodic configuration given
in Section~\ref{configs}, we prove that the flow is periodic in the direction
$\theta$ on $Y$. 
\par
In Section~\ref{sec:examples}, we check the hypothesis of 
Theorem~\ref{closure} for  an infinite family of examples  
arising from Thurston's construction.
\par
In the last section, we briefly discuss the existence of a
$2T_{fix}2C$ splitting on surfaces belonging to $\LLL$.
\end{Reader's guide}

\begin{Acknowledgments}
We thank C.~McMullen for comments on preliminary versions of this article
and E.~Nipper for a careful reading.
This work was partially supported by the ANR Teichm\"uller ``projet blanc'' ANR-06-BLAN-0038.
\end{Acknowledgments}

\section{Background}
\label{background}

In this section we review basic notions concerning translation
surfaces, trace fields, $J$-invariant and SAF-invariant.

A {\em translation surface} (or {\em flat surface}) is a
Riemann surface $X$ with finitely many {\em singularities} $P_i$, 
plus the choice of charts covering $X \setminus \cup \{P_i\}$, such
that the transition functions are translations. Equivalently, a translation
surface is given by a pair $(X,\omega)$ of a Riemann surface $X$ and 
a holomorphic one-form $\omega$. For {\em half-translation
surfaces} the conditions are relaxed to demanding the transition
functions to be $\pm {\rm id}$ composed with a translation. 
Equivalently, half-translation surfaces are given by a pair $(X,q)$ 
of a Riemann surface $X$ and a quadratic differential $q$.
\medskip

A half-translation surface admits a ramified double covering,
which is a translation surface and except for the introduction
we will work exclusively with translation surfaces.
Translation surfaces correspond bijectively to pairs $(X,\omega)$
of a Riemann surface $X$ and a holomorphic one-form. Similarly,
half-translation surfaces correspond bijectively to pairs $(X,q)$,
where $q$ is a quadratic differential. See e.g.\ \cite{MT} for
a survey on these notions. The singularities correspond to
the zeros of $\omega$ resp.\ of $q$ under this bijection. \medskip

There is a natural action of $\GL_2^+(\R)$ on (half-) translation
surfaces by post-composing the integration charts with the corresponding linear
map. This action respects the number and multiplicities of zeros,
called the signature,
of the one-from (resp.\ the quadratic differential). Consequently
the moduli space of pairs  $(X,\omega)$ (respectively of pairs $(X,q)$),
denoted by $\HHH$ or also by $\Omega M_g$ (respectively $\QQQ M_g$)
is stratified by the signature. The stratum
$\HHH(2,2)$ has two connected components (see \cite{KZ}). One
component is the hyperelliptic component. The locus $\LLL$ is a
codimension 1 subspace of the non hyperelliptic component of the
stratum $\HHH(2,2)$. \medskip

Straight lines in the translation charts are geodesics for
the metric $|\omega|$. A maximal subset of $X$ swept out by parallel 
geodesics is called an (open) {\em cylinder}. Its closure will
be bounded by a finite number of {\em saddle connections}, 
geodesics joining the singularities.
We will say that a cylinder is {\it simple} if each of its boundaries
consists of a unique saddle connection (joining possibly the two
zeroes). A geodesic has a well
defined direction in $\C \cong \R^2$ and the direction of a
cylinder is the direction of any of its geodesics. \medskip

Consider all geodesics in a fixed direction $\theta$. This
direction is called {\em periodic} if there is an open cylinder
in this direction. $\theta$ is called {\em completely
periodic}, if $X$ decomposes completely into cylinders and
saddle connections in this direction.
Note that there is also the notion (\cite{C}) of a {\em completely periodic
surface $(X,\omega)$}, a surface such that each direction that
has a cylinder is automatically completely periodic. This notion
will not be considered in the sequel of this paper. \medskip

A completely periodic direction is called {\em parabolic}, if
the moduli of the cylinders in this direction are commensurable.
Parabolic directions are important, since a composition of
suitable powers of Dehn twists along the cylinders produces
a diffeomorphism which is affine with respect to the charts
given by $\omega$ (\cite{Vee89}). We denote by $\Aff^+(X,\omega)$
the group of orientation-preserving affine diffeomorphisms and
by $\SL(X,\omega)$ its image under the natural map to $\SL_2(\R)$.  
This image is called {\em affine group} or {\em Veech group} of
$(X,\omega)$. It is well-known, that a diffeomorphism in 
$\Aff^+(X,\omega)$ is pseudo-Anosov, if and only if its
image in $\SL(X,\omega)$ is a hyperbolic element of $\SL_2(\R)$. 
The {\em trace field} of $(X,\omega)$ is the field generated over
$\Q$ by the traces of all elements in $\SL(X,\omega)$. One can also
define the trace field by the following way. One defines the {\it
holonomy vectors} to be the integrals of 
$\omega$ along the saddle connections. Let us denote by 
$\Lambda=\Lambda(\omega)$ the subgroup of $\mathbb R^2$ generated
by holonomy vectors $\Lambda = \int_{H_1(X,\mathbb Z)} \omega$. 
If $e_1,e_2 \in \Lambda$ are two nonparallel vectors in $\mathbb
R^2$, one defines the {\it holonomy field} $k$ to be the smallest
subfield of $\mathbb R$ such that every element of
$\Lambda$ may be written as $ae_1+be_2$ with $a,b \in k$. It is known
(\cite{KS} Theorem~28) that if $\SL(X,\omega)$ contains a
psuedo-Anosov diffeomorphism, then the trace field of $(X,\omega)$ 
coincides with $k$. In particular, any direction $\theta$ of saddle connection
belongs to the trace field. \medskip

If $P$ is a polygon in $\R^2$ with vertices $v_1,\dots,v_n$ in
counterclockwise order about the boundary of $P$, then the
$J$-invariant of $P$ is $J(P)=\sum_{i=1}^n v_i \wedge v_{i+1}$ (with
the dummy condition $v_{n+1}=v_1$). Here $\wedge$ is taken to mean
$\wedge_\Q$ and $\R^2$ is viewed as a $\Q$-vector space. $J(P)$ is a
translation invariant (e.g. $J(P+\overrightarrow{v})=J(P)$), thus this
permits to define $J(X,\omega)$ by $\sum_{i=1}^k J(P_i)$ where $P_1 \cup \dots
\cup P_k$ is a cellular decomposition of $(X,\omega)$ into planar
polygons (see~\cite{KS}).
\par
We will also make use of the SAF-invariant of an interval exchange
transformation $f$ as follows. We define a linear projection $J_{xx}:\R^2
\wedge_\Q \R^2 \rightarrow \R \wedge_\Q \R$ by
$$
J_{xx} \left( \left( \begin{smallmatrix} a \\ b\end{smallmatrix}\right)
  \wedge  \left( \begin{smallmatrix} c \\ d\end{smallmatrix}\right)
    \right) = a \wedge c.
$$
If $f$ is an interval exchange transformation induced by the first
return map of the vertical foliation on $(X,\omega)$ (on a transverse
interval $I$) then let us define the SAF-invariant of $f$ by $SAF(f) =
J_{xx}(X,\omega)$. Note that the 
definition does not depend of the choice of $I$ if the interval meets
every vertical leaf (see~\cite{Ar1}). We will also say that $SAF(f)$
is the SAF-invariant of $(X,\omega)$ in the vertical direction. We define in an
obvious manner the SAF-invariant of $(X,\omega)$ in a direction of the trace field.
\par 
If $f$ is a periodic interval exchange transformation then
$SAF(f)=0$. The converse is not true in general. But if $f$ is defined
over $2$ or $3$ intervals and $SAF(f)=0$ then $f$ is periodic.\medskip

We now sketch the Thurston construction (\cite{Thurston}, 
see also \cite{Vee89})
of surfaces\footnote{Also
known as {\em bouillabaisse surfaces}, after a talk by J.Hubbard, CIRM 2003}
with pseudo-Anosov diffeomorphisms. Choose a 
pair $(\{\gamma_i\}_{i\in I}, \{\gamma_j\}_{j\in J})$ 
of multi\-curves, i.e.\ of sets of simple closed curves on $X$ such that 
$X \setminus (\cup_{i\in I} \gamma_i \cup  \cup_{j \in J} \gamma_j)$ is
a union of simply connected regions. Let $M_{r,s} = (\gamma_r, \gamma_s)$
for $r,s \in I \cup J$ be the symmetric intersection matrix.
Choose integer weigts $m_r$ for $r \in I \cup J$ and let
$(h_r)$ be the unique (by Perron-Frobenius) positive eigenvector
satisfying
$$ \mu h_r = \sum_{s \in I \cup J} m_r M_{rs} h_s.$$
Now glue rectangles $[0,h_r] \times [0,h_s]$ according
to the intersection pattern of the multicurves $\gamma_i$, $i\in I$,
and $\gamma_j$, $j \in J$,  to
obtain a closed surface (with cone-type singularities arising
from the corners).
\par
Suitable powers of Dehn twist along the curves $\gamma_i$, $i\in I$ 
(resp.\ along the curves $\gamma_j$, $j \in J$) define two non-commuting 
parabolic elements
in $\SL(X,\omega)$. Suitable products of them give hyperbolic
elements in $\SL(X,\omega)$, thus the corresponding diffeomorphism
is a pseudo-Anosov diffeomorphism. One of the main results of \cite{HuLa} shows,
that not all pseudo-Anosov diffeomorphisms arise in this way.

\section{Configurations}
\label{configs}

In this section, we classify (topologically) the configurations of
completely periodic directions $\theta$ on a translation surface belonging to
the hyperelliptic locus~$\LLL$.

\subsection{Statement of the result}

Cut $X$ along the set of saddle connections in the direction $\theta$. 
The result is a set of cylinders. To reconstruct the surface $(X,\omega)$, 
we have to glue these cylinders according to a combinatorics which encode which
part of a cylinder is glued to another part of a cylinder. We will
call such a combinatorics by a {\it configuration}. For instance, in
Figure~\ref{cap:conf:list}a, the configuration presented has $3$ cylinders (for
the vertical direction $\theta=\pi/2$). One cylinder is fixed by the involution (we
have represented the Weierstrass points by small bullets) and the two
others are exchanged. We label the intervals representing the same
saddle connection on the surface by the same number. Horizontal saddle connections
are identified by vertical translation. A simple computation, using the 
Euler characteristic, shows that the number of cylinders is bounded
above by $4$. In this section we will prove the following result.

\begin{Theorem}
\label{cylclassH}
Let $(X,\omega)$ be a translation surface in the hyperelliptic locus
$\LLL$. Let us assume that the vertical direction is completely
periodic. Then all possible configurations of cylinders are prescribed by 
Figure~\ref{cap:conf:list}. More precisely:

\begin{enumerate}

\item There are four configurations with two cylinders:
\begin{itemize}
\item One configuration with two exchanged cylinders
corresponding to Figure~\ref{cap:conf:list}h.

\item Three configurations with two fixed cylinders
  corresponding to Figure~\ref{cap:conf:list}i, 
  Figure~\ref{cap:conf:list}j and Figure~\ref{cap:conf:list}k.

\end{itemize}

\item There are four configurations with three cylinders:
\begin{itemize}
\item Two configurations with one fixed cylinder
  corresponding to Figure~\ref{cap:conf:list}a and
  Figure~\ref{cap:conf:list}b.

\item Two configurations with three fixed cylinders
  corresponding to Figure~\ref{cap:conf:list}c and 
  Figure~\ref{cap:conf:list}d.

\end{itemize}

\item There are three configurations with four cylinders:
\begin{itemize}
\item Two configurations with two simple cylinders
corresponding to Figure~\ref{cap:conf:list}e and 
  Figure~\ref{cap:conf:list}f.

\item One configuration with one simple cylinder
  corresponding to Figure~\ref{cap:conf:list}g.

\end{itemize}
\end{enumerate}
\end{Theorem}

\begin{samepage}
\begin{figure}[htbp]
\begin{center}
\psfrag{1}{$\scriptstyle 1$} \psfrag{2}{$\scriptstyle 2$}
\psfrag{3}{$\scriptstyle 3$} \psfrag{4}{$\scriptstyle 4$}
\psfrag{5}{$\scriptstyle 5$} 
  \subfigure[]{\epsfig{figure=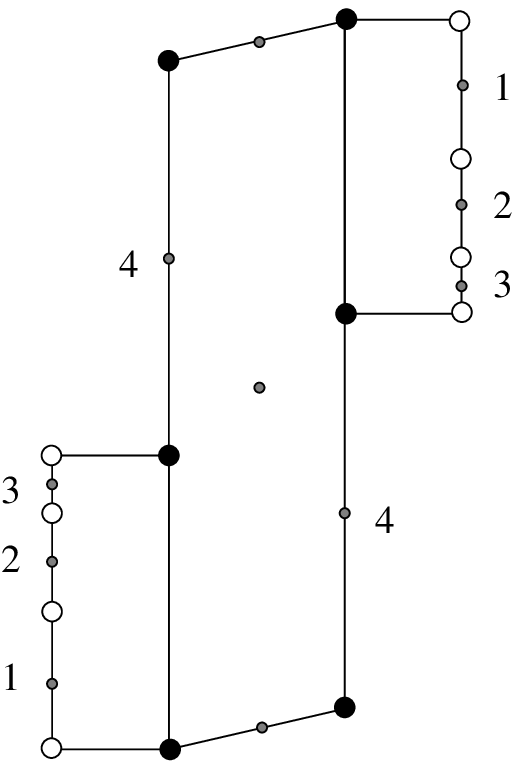,width=2.7cm}} \qquad \qquad
  \subfigure[]{\epsfig{figure=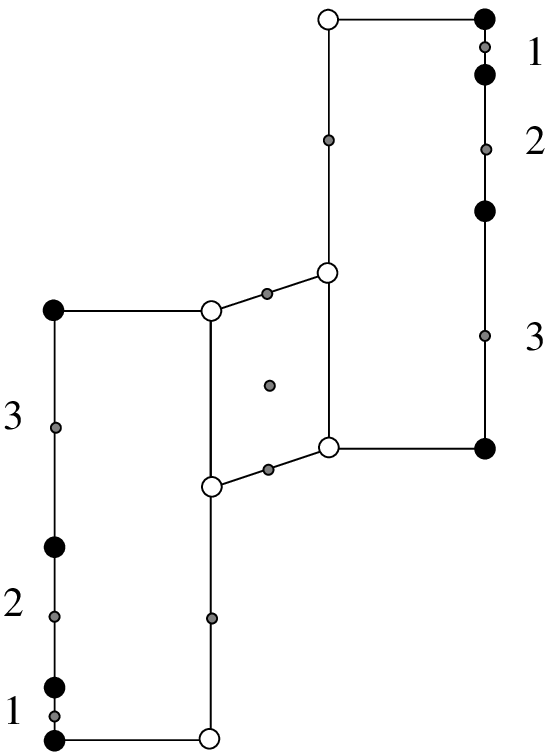,width=2.7cm}}\qquad \qquad
  \subfigure[$\ 3C_\fix$]{\epsfig{figure=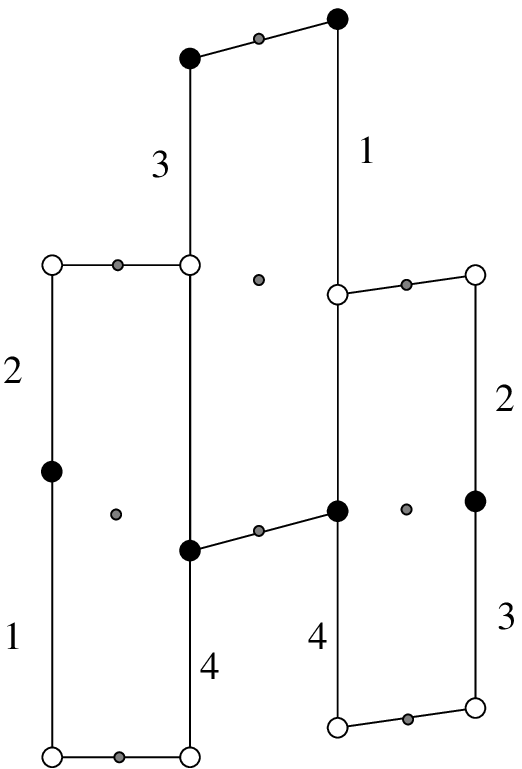,width=2.7cm}}\qquad \qquad
   \subfigure[\ $3C$]{\epsfig{figure=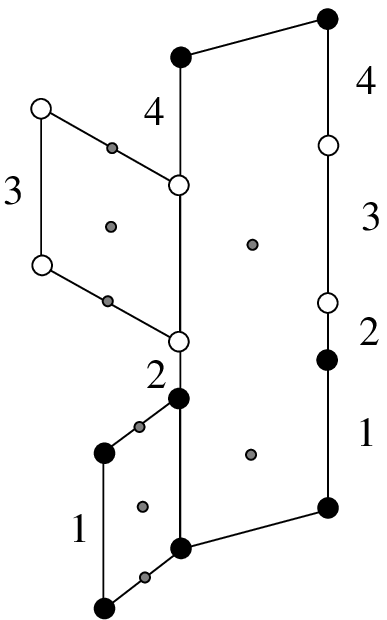,width=1.8cm}}

  \subfigure[$\ 2T_\fix2C$]{\epsfig{figure=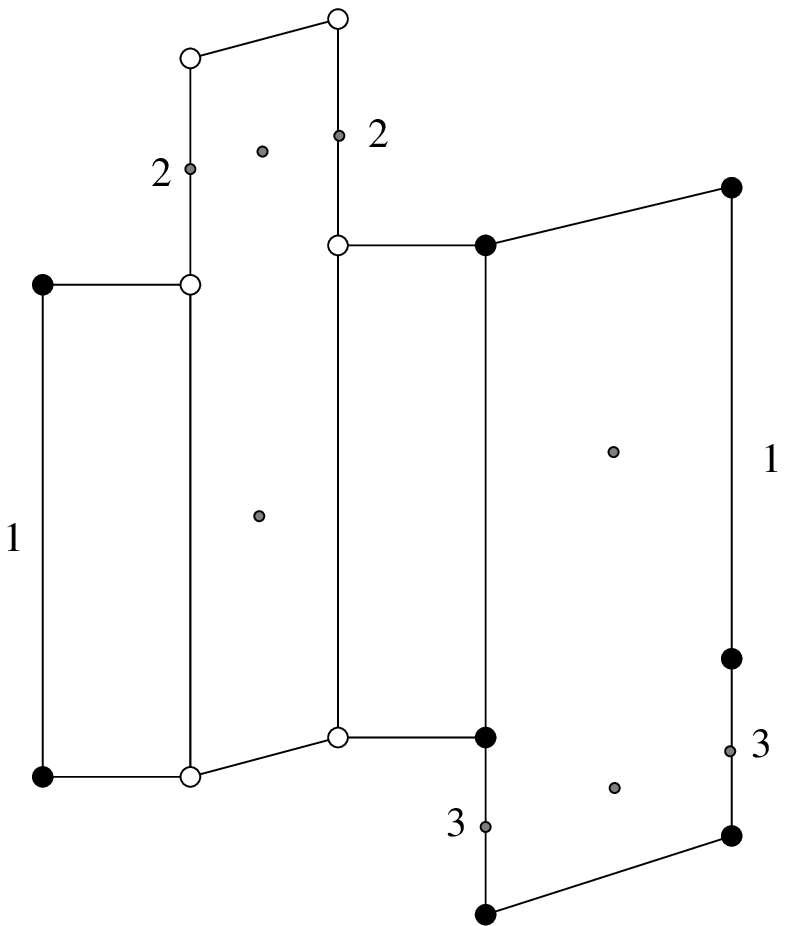,width=3.2cm}}\qquad \qquad
  \subfigure[$\ 2T2C_\fix$]{\epsfig{figure=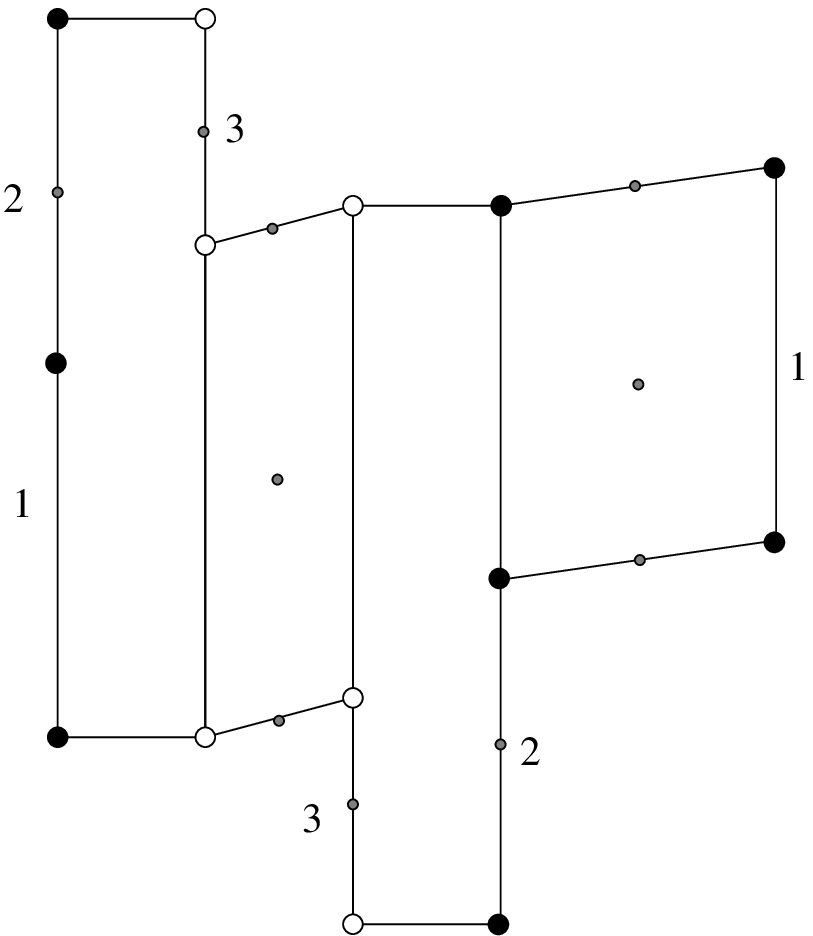,width=3cm}}\qquad \qquad
  \subfigure[]{\epsfig{figure=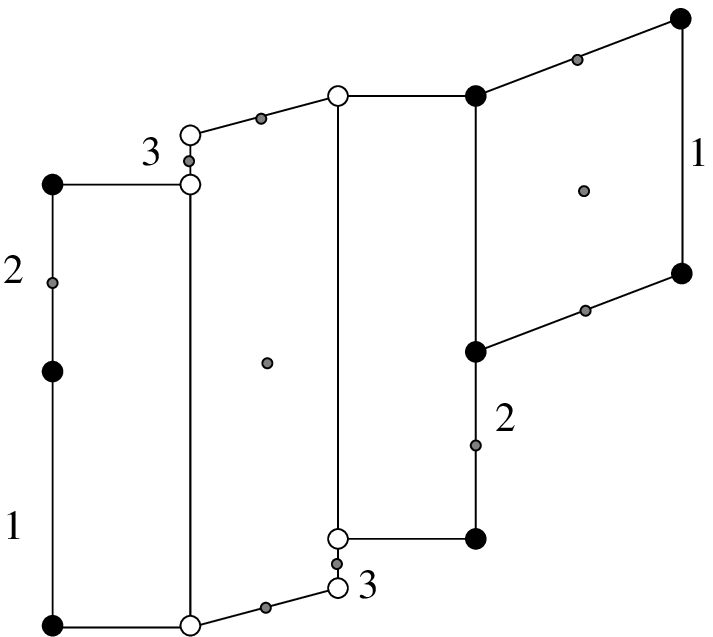,width=3cm}}

  \subfigure[]{\epsfig{figure=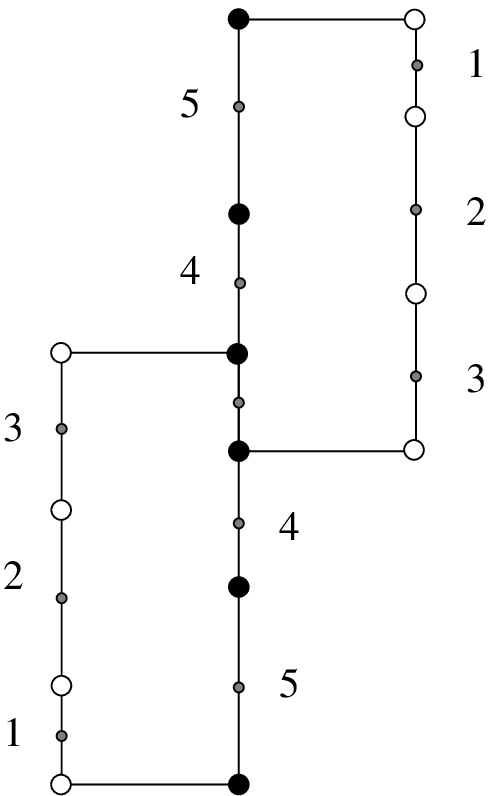,width=2.4cm}}\qquad \qquad
  \subfigure[]{\epsfig{figure=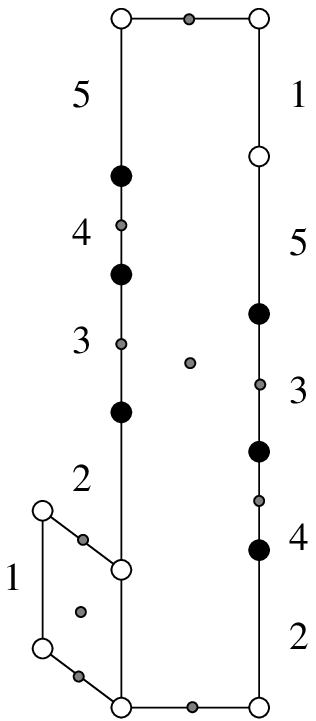,width=1.9cm}}\qquad \qquad
  \subfigure[]{\epsfig{figure=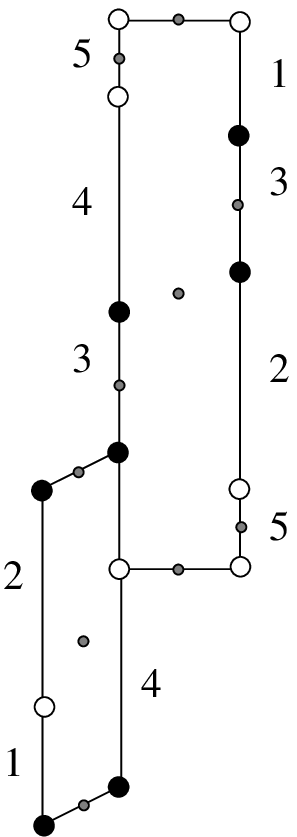,width=1.9cm}}\qquad \qquad
  \subfigure[]{\epsfig{figure=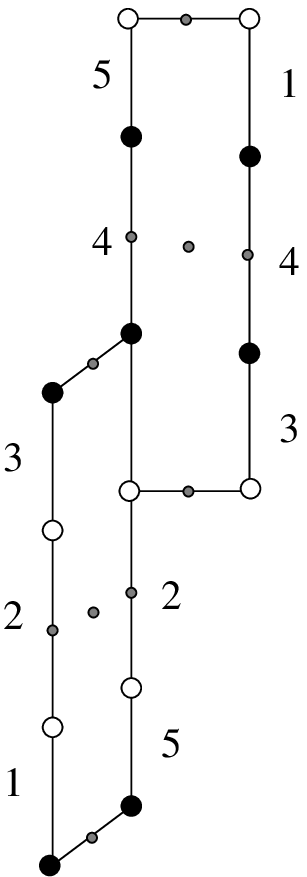,width=1.9cm}}
\end{center}
\caption{
List of (topological) configurations of completely periodic surfaces
in~$\LLL$ (for the vertical direction).
}\label{cap:conf:list}
\end{figure}
\end{samepage}

n\begin{Remark}
If $(X,\omega)$ admits a one cylinder decomposition then
Thurston's construction implies that either
$\textrm{SL}(X,\omega)$ is cyclic, generated by a parabolic element,
or $\textrm{SL}(X,\omega)$ is commensurable to $\textrm{SL}_2(\Z)$,
e.g. $(X,\omega)$ is a covering of at orus ramified at over point (see
Lemma~\ref{lm:bound:cylinders}).
\end{Remark}

We will use the following obvious lemma.

\begin{Lemma}
\label{cor:topology}
Let $(X,\omega)\in \LLL$ be a translation surface and let us assume
that the vertical direction is completely periodic direction. Then:

\begin{enumerate}

\item There exist $6$ saddle connections joining a zero to another
zero.

\item A Weierstrass point is located either on the middle of a saddle
joining a zero to itself or on the core of a cylinder. In the
last case there are exactly two Weierstrass points on the core curve.

\item
\label{cor:topology:item1}
If no cylinder is fixed by the involution then there is no saddle
connection connecting the two zeroes.

\item
\label{cor:topology:item2}
If exactly one cylinder is fixed by the involution then there is at most one
saddle connection connecting the two zeroes. 

\end{enumerate}
\end{Lemma}

We will now give a proof of the main theorem of this section. 

\subsection{Proof of Theorem~\ref{cylclassH} in case of four cylinders}

Let $(X,\omega)$ be a translation surface and let us assume that the
vertical direction is completely periodic and decomposes the surface
in four metric cylinders. Clearly the number of fixed
cylinders (with respect to the hyperelliptic involution) is $0$ or
$2$. Indeed if all the cylinders are fixed then each cylinder will 
possess two fixed points (for the involution). Recall that the two
zeroes are fixed, thus the hyperelliptic involution should possess
$4\times 2+2=10$ fixed points which is impossible. \medskip

\begin{Claim}
There are exactly two fixed cylinders.
\end{Claim}

\begin{proof}[Proof of the Claim]
Assume not. Then from above discussion, there is no fixed
cylinder. Thus all the Weierstrass points (except for the two
singularities) are located on the middle of the saddle connection. In
particular, there is no saddle connection connecting the two zeroes.

Note that a Weierstrass point cannot be located on the boundary of a
simple cylinder. Therefore there is no simple cylinder on such a
configuration. Thus the only possible configuration is the following
one: four cylinders, each of them having three saddle connections on
its boundaries. It is then not hard to check that there is no possible
configuration with $8$ fixed points for the involution. The claim is
proven.
\end{proof}

Thus, a configuration with four cylinders must have exactly two fixed
cylinders.

\begin{Claim}
The number of simple cylinders is $1$ or $2$.
\end{Claim}

\begin{proof}[Proof of the Claim]
Assume not. If there are four simple cylinders then obviously
$(X,\omega)$ is a torus, which is impossible. \\
If there is no simple cylinder then the number of saddle connections
is at least $8$ for the two fixed cylinders and $6$ for the other two.
Finally the number of saddle connections is at least $8+6=14>12$, which
is a also a contradiction.
\end{proof}

Hence there is either $1$ (necessarily fixed) simple cylinder, or $2$ (fixed
or exchanged) simple cylinders. These three possibilities lead to
three configurations. It is not hard to show that the three possible
cases are respectively given by
Figure~\ref{cap:conf:list}g-\ref{cap:conf:list}f-\ref{cap:conf:list}e.
The last two configurations will be respectively called $2T 2C_\fix$
and $2T_\fix 2C$.

\subsection{Proof of Theorem~\ref{cylclassH} in case of three cylinders}

Let $(X,\omega)$ be a translation surface and let us assume that the
vertical direction is completely periodic and decomposes the surface
in three metric cylinders. Clearly the number of fixed
cylinders (with respect to the hyperelliptic involution) is $1$ or $3$.

\begin{Claim}
Let us assume that only one cylinder is fixed by the involution. Then 
the number of simple cylinders is $0$ or $1$.
\end{Claim}

\begin{proof}[Proof of the Claim]
Suppose not. Then the configuration for the vertical direction must
contain two or three simple cylinders. In the last case $X$ is a
torus, which is impossible.

If the configuration contains only two simple cylinders, they are
necessarily exchanged. Thus each boundary of the fixed cylinder must
possess exactly four saddle connections. Note that a Weierstrass
point can not be located on the boundary of a simple cylinder. Let us
count the number of Weierstrass points:
\begin{enumerate}
\item the two zeroes.
\item two Weierstrass points on the core curve of the fixed cylinder.
\item at most two Weierstrass points on the middle of the saddle
  connections of the fixed cylinder.
\end{enumerate}

Thus there are at most $6<8$ Weierstrass points which is a contradiction.
\end{proof}

Hence for such a configuration, there is no simple cylinder or one
fixed simple cylinder. It is then not difficult to check that theses
two cases lead respectively to the configurations presented in
Figure~\ref{cap:conf:list}a and in Figure~\ref{cap:conf:list}b. 

\begin{Claim}
Let us assume that the three cylinders are fixed by the involution. 
Then the number of simple cylinders is $0$ or $2$.
\end{Claim}

\begin{proof}[Proof of the Claim]
Assume not. Then the configuration contains one fixed simple
cylinder. By a straightforward calculation, one proves that no
possible configuration can occur in the hyperelliptic locus $\LLL$.
\end{proof}

Thus in that case, the configuration possesses either three (non
simple) fixed cylinders or two simples and one other (all fixed)
cylinders. It is then easy to see that these two cases lead
respectively to configurations presented in
Figure~\ref{cap:conf:list}c and in
Figure~\ref{cap:conf:list}d. \medskip

The proof of the theorem in case of two
cylinders is left  to the reader as an exercise.

\section{Complete periodicity}
\label{sec:completely:periodic}

We recall the following result by H. Masur (see \cite{Ma,MT}):

\begin{Theorem} [Masur]
Let $(X,\omega)$ be a translation surface.  There is a dense set of periodic directions on $(X, \omega)$.
\end{Theorem}

In this section, we prove the following strengthenings for translation
surfaces in the locus $\LLL$:

\begin{Theorem} 
\label{periodic}
Let $(X,\omega)$ be a translation surface that belongs to the
hyperelliptic locus $\LLL$. There is a dense set of periodic
directions $\theta \in \SU^1$ such that, in the direction $\theta$,
$\omega$ has a metric cylinder containing a Weierstrass point.
\end{Theorem}

\begin{Theorem} 
\label{completely-periodic}
Let $(X,\omega)$ be a translation surface in the hyperelliptic locus
$\LLL$. Let us assume that the Teichm\"uller disc of  $(X,\omega)$ is
stabilized by a pseudo-Anosov diffeomorphism and that there exists a completely periodic 
direction. Then every periodic
direction that contains a cylinder invariant by the hyperelliptic
involution is completely periodic. Moreover the set of completely
periodic directions is dense in $\SU^1$.
\end{Theorem}

\begin{Remark}
Theorem~\ref{completely-periodic} immediately
implies Theorem~\ref{complete-periodicity}.

 The statement of Theorem \ref{completely-periodic}
 is
a weak form of the completely periodicity studied by Calta in genus 2
(see~\cite{C}). \\
Nevertheless, it is not true that the hypothesis of the theorem
implies that the surface is completely periodic in the sense of Calta.
A direction with a splitting of the surface  into  two cylinders and
two non periodic tori was exhibited on the Arnoux-Yoccoz example see
\cite{HLM}. Similar constructions can be done for surfaces
arising from Thurston's construction.
\end{Remark}
\par
The following obvious remark shows that the behaviour stated
in Theorem~\ref{completely-periodic} is indeed special.
\par
\begin{Proposition} \label{noncpgeneric}
Almost every surface in any stratum does not admit a completely
periodic direction.
\end{Proposition}
\par
\begin{proof}
Complete periodicity can be expressed via proportionality of
a non-empty set of relative periods. Since any stratum admits
a coordinate system given by integration of a basis of relative periods,
surfaces with a completely periodic direction are 
given by a linear subspace of positive codimension in this coordinate
system.
Since there are only countably many choices of such a basis,
the set in question is of Lebesque measure zero.  
\end{proof}
\subsection{Proof of Theorem \ref{periodic}}

\begin{Definition} 

A {\em slit} is a geodesic segment embedded on a
translation surface (in particular, it has no self-intersection and
its end points differ). A {\em slit torus} is a flat torus with a marked
point where a slit starting from the marked point is removed. 

\end{Definition}

The key ingredient in the proof is the following lemma.

\begin{Lemma} 
\label{slit-torus}
Let $(X,\omega)$ be a translation surface of area 1 in the
hyperelliptic locus $\LLL$. Let us assume that the vertical
direction is  periodic, and, let $L$ be the length of the shortest
vertical saddle connection. There is a direction $\theta$ with
$\vert \cos(\theta) \vert \leq 1/L^2$ such that, in the direction
$\theta$, the surface $(X,\omega)$ has a metric cylinder containing a Weierstrass
point.
\end{Lemma}

\begin{proof}
If a vertical cylinder is invariant by the hyperelliptic involution,
there is nothing to prove. Otherwise, the vertical direction
contains two cylinders $C_1$ and $C_2$ exchanged by the involution.
\par
Section~\ref{configs} describes the completely periodic vertical
directions. In fact, it also gives all the types of periodic
directions. By rescaling, we may suppose that the integral along
the core curve of $C_1$ (and hence also of $C_2$) is in $\Q\cdot i$,
while the width is in $\Q$. Hence there is a surface $(Y,\eta)$
arbitrarily close to $(X,\omega)$ with cylinders $C_i$ of the same
size as those on $X$ and with relative periods in $\Q[i]$. The
surface $(Y,\eta)$ is square-tiled by construction. Hence its
vertical direction is completely periodic. A path between 
$(X,\omega)$ and $(Y,\eta)$ fixing $C_1$ and $C_2$ gives a way
of describing $(X,\omega)$ as a deformation of a
completely periodic surface, which leaves the cylinders $C_1$
and $C_2$ untouched.

The periodic configurations with two cylinders exchanged by the hyperelliptic
involution are listed in Figure~\ref{cap:conf:list}.

Using this list, we remark that $X \setminus (C_1 \cup C_2)$ falls
into the following list:
\begin{enumerate}
\item a torus cut along a slit  (Figure \ref{cap:conf:list}a)
\item a periodic cylinder (Figure \ref{cap:conf:list}b)
\item a union of two tori $T_1$ and $T_2$ cut along vertical slits
(Figure \ref{cap:conf:list}e)
\item a union of two periodic cylinders (Figure \ref{cap:conf:list}f)
\item a slit torus and a periodic cylinder (Figure \ref{cap:conf:list}g)
\item empty (Figure \ref{cap:conf:list}h).
\end{enumerate}

All the connected components of $X \setminus (C_1 \cup C_2)$ are
invariant under the hyperelliptic involution. Thus the vertical
direction contains a periodic cylinder fixed by the involution when
$X \setminus (C_1 \cup C_2)$ is a cylinder (case~(2)) or a union of a
cylinder and a torus (case~(5)).

Consequently, we assume that $X \setminus (C_1 \cup C_2)$ is the union
of one or two slit tori fixed by the hyperelliptic involution. Let
$T$ be one of these tori. We call {\em the elliptic involution} 
of a torus $T$ embedded into a flat surface of higher genus $X$ an 
elliptic involution that fixes one of the singularities of $X$ that
lies on $T$.

\begin{Claim}
Suppose that $T$ is a slit torus and that the direction of the slit
does not contain a lattice vector of $\Lambda$ where $T \cong \C/\Lambda$.
Then $T$ contains a metric cylinder that does not intersect
the slit and that is invariant under the hyperelliptic involution.
\end{Claim}

\begin{proof}[Proof of the Claim]
Up to normalization, the lattice is
$$\Lambda = \left\langle 
\begin{pmatrix}
 1 \\
 0
 \end{pmatrix},
\begin{pmatrix}
 x \\
 y
 \end{pmatrix} \right\rangle$$
and slit is vertical. By our hypothesis, we may assume $x \in (0,1)$.
As shown in Figure~\ref{cap:slit-torus}, the direction $\begin{pmatrix}
 x \\
 y
 \end{pmatrix}$
 is periodic and the cylinder contains a Weierstrass point. 
Moreover, the periodic leaves of the cylinder do not intersect the slit.
\begin{figure}[htbp]
\label{fig:slit-torus}
\begin{center}
\psfrag{s}{\textrm{slit}}
\psfrag{xy}{$\left(\begin{smallmatrix}x\\y \end{smallmatrix}\right)$}
\psfrag{1}{$\left(\begin{smallmatrix}1\\0 \end{smallmatrix}\right)$}
\psfrag{0}{$\left(\begin{smallmatrix}0\\0 \end{smallmatrix}\right)$}
\includegraphics[width=2.5cm]{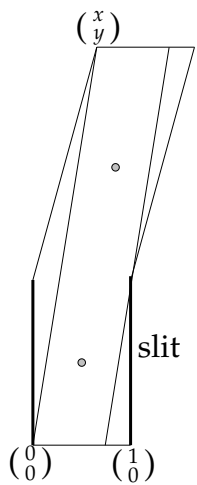}
\end{center}
\caption{
A cylinder fixed by the elliptic involution on a slit
torus}\label{cap:slit-torus}
\end{figure}
\end{proof}

The lengths of each slit is at least $L$ and the area of $T$ is less than 1.
Thus the distance between the left and right boundaries of $T$ is at
most $1/L$ and the length of the cylinder is at least $L$. This yields
$\vert \cos(\theta) \vert \leq 1/L^2$.

It remains to consider the case described in
Figure~\ref{cap:conf:list}h when $X \setminus (C_1 \cup C_2)$ is
empty. There is a segment starting from the Weierstrass point which is
in the middle of the saddle connection labelled by 3 to the middle of
the saddle connection labelled by 5. This segment and its image under
the hyperelliptic involution form a closed geodesic $\gamma$ in a
direction $\theta$. The periodic geodesic $\gamma$ is the core of a
cylinder (containing a Weierstrass point) and the angle $\theta$
satisfies $\vert \cos(\theta) \vert \leq 1/L^2$. This ends the proof of
Lemma~\ref{slit-torus}.
\end{proof}

\begin{proof}[Proof of Theorem~\ref{periodic}]
We assume that the area of $(X, \omega)$ is equal to one and we fix
$\varepsilon >0$ and $\phi \in \SU^1$. Applying Masur's theorem, there
exists a periodic direction $\psi$ with
$\vert \psi - \phi \vert < \varepsilon/2$.
On the other hand, for each $\delta$, there is only
finite number of
directions with a
saddle connection of length less than $\delta$. We may thus assume,
without loss of
generality, that the length of the shortest saddle connection in the
direction $\psi$ is at least 
$\frac{1}{\sqrt{\sin(\varepsilon/2)}}$. By
Lemma~\ref{slit-torus}, there is a periodic direction $\theta$ with a
cylinder containing a Weierstrass point and $\vert \sin(\theta - \psi)
\vert < \sin(\varepsilon/2)$ for $\varepsilon$ small enough. Thus $\vert \phi
- \theta \vert < \varepsilon$. Theorem~\ref{periodic} is proven.
\end{proof}

\subsection{Proof of Theorem \ref{completely-periodic}}

We first recall some facts concerning the $J$-invariant of Kenyon-Smillie~\cite{KS}
and the SAF-invariant.

\begin{Theorem}[Calta-Smillie \cite{CaSm} Theorem~1.5]
\label{theo:CaSm}
Assume that $(X,\omega)$ is stabilized by a pseudo-Anosov
diffeomorphism and that there exists a completely periodic direction. 
Then the SAF-invariant vanishes for all directions of the trace field.
\end{Theorem}

The following facts concerning the SAF-invariant  are essentially
contained in Arnoux's thesis  \cite{Ar1} and  were used by Calta
\cite{C} and also by McMullen \cite{Mc2} (who works with the flux
instead of the SAF-invariant).

\begin{Lemma} \label{H(2)}
Let $(Y,\alpha)$ be a translation surface in the stratum $\HHH(2)$
such that the $SAF$-invariant vanishes in the vertical direction.
Then the vertical direction is completely periodic.
\end{Lemma}

\begin{Lemma} \label{H(1,1)}
Let $(Y,\alpha)$ be a translation surface in the stratum $\HHH(1,1)$
such that the $SAF$-invariant vanishes in the vertical direction. If
there are two non homologous vertical saddle connections, then the
vertical direction is completely periodic.
\end{Lemma}

\begin{proof}[Proof of Lemmas~\ref{H(2)},~\ref{H(1,1)}]

For the sake of the completeness, we give a sketch of proof of Lemmas
\ref{H(2)}, \ref{H(1,1)}. We recall that, in genus 2, if the
SAF-invariant vanishes then the flow in the vertical direction is
not minimal [Ar]. Therefore, there is a vertical saddle connection.

If $(Y,\alpha)$ belongs to $\HHH(2)$,  the induced map of the vertical flow on a suitable
transversal is an interval exchange
transformation on 4 intervals. Since there is a saddle connection, it
reduces to an interval exchange transformation on 3 intervals.  An
interval exchange on 3 intervals with SAF-invariant equals to 0 is
periodic (see \cite{Ar1}). Thus, the vertical flow on $(Y,\alpha)$
is periodic.

If $(Y,\alpha)$ belongs to $\HHH(1,1)$,  the induced map of the vertical flow on a
suitable transversal $\tau$ is an interval exchange
transformation on 5 intervals. By hypothesis, there are two vertical
saddle connections  $\gamma_1$ and $\gamma_2$ such that $Y \setminus
(\gamma_1 \cup \gamma_2)$ is connected. Thus, the induced map on
$\tau$ reduces to an interval exchange transformation on 3 intervals with
SAF-invariant equals to 0. This interval exchange transformation is
periodic and thus vertical flow is periodic.

\end{proof}

\begin{proof}[Proof of Theorem~\ref{completely-periodic}]

Let us consider a cylinder $C$ fixed by the hyperelliptic involution
on $(X,\omega)$. Let $\theta$ be the direction of $C$. It is a
direction of the holonomy field (see background) thus, by
Calta-Smillie's result, the SAF-invariant vanishes in the direction
$\theta$. Without loss of 
generality we may assume that the direction $\theta$ is the vertical
direction. The complement of $C$ is a slit surface $Y$ with vertical
boundaries. We have to prove that the vertical flow is completely
periodic on $Y$. The SAF-invariant equals to 0 on $C$, thus it
vanishes on $Y$. The topological type of $Y$ can be deduced from the
description of  the completely periodic directions (see
Figure~\ref{cap:conf:list}) by the same argument as in Lemma~\ref{slit-torus}.

There is a canonical way to associate to $Y$ a compact surface $\hat
Y$ without boundary. $Y$ is obtained from $X$ by removing a cylinder $C$
invariant under the hyperelliptic involution. A slit on $Y$
corresponds to a saddle connection  on the boundary of $C$ and that is
not invariant under the hyperelliptic involution. Since $C$ is fixed by
the hyperelliptic involution, these slits come in pairs. $\hat{Y}$
is obtained from $Y$ by gluing the pairs of slits exchanged by the
hyperelliptic involution. The vertical flow on $Y$ is periodic if
and only if it is periodic on $\hat{Y}$.

Referring to the Figure~\ref{cap:conf:list}, we denote the cylinders by I, II,
III, etc from the left to the right.

\begin{itemize}
\item $\hat{Y}$ belongs to $\HHH(2)$, if II is removed in Figure~\ref{cap:conf:list}a.
\item $\hat{Y}$ belongs to $\HHH(2)$, if II is removed in Figure~\ref{cap:conf:list}b.
\item  $\hat{Y}$ belongs to $\HHH(1,1)$, if I is removed in Figure~\ref{cap:conf:list}c.
\item  $\hat{Y}$ belongs to $\HHH(1,1)$, if II is removed in Figure~\ref{cap:conf:list}c.
\item  $\hat{Y}$ belongs to $\HHH(1,1)$, if III is removed in Figure~\ref{cap:conf:list}c.
\item $\hat{Y}$ belongs to $\HHH(2)$, if I is removed in Figure~\ref{cap:conf:list}d.
\item $\hat{Y}$ belongs to $\HHH(2)$, if II is removed in Figure~\ref{cap:conf:list}d.
\item  $\hat{Y}$ is the union of two cylinders,  if III is removed in 
Figure~\ref{cap:conf:list}d.
\item  $\hat{Y}$ belongs to $\HHH(2)$, if II is removed in Figure~\ref{cap:conf:list}e.
\item $\hat{Y}$ belongs to $\HHH(2)$, if IV is removed in Figure~\ref{cap:conf:list}e.
\item $\hat{Y}$ belongs to $\HHH(2)$, if II is removed in Figure~\ref{cap:conf:list}f.
\item $\hat{Y}$ belongs to $\HHH(2)$, if IV is removed in Figure~\ref{cap:conf:list}f.
\item $\hat{Y}$ belongs to $\HHH(2)$, if IV is removed in Figure~\ref{cap:conf:list}g.
\item $\hat{Y}$ belongs to $\HHH(2)$, if II is removed in Figure~\ref{cap:conf:list}g.
\item $\hat{Y}$ belongs to $\HHH(2)$, if I is removed in Figure~\ref{cap:conf:list}i.
\item $\hat{Y}$ is a cylinder, if II is removed in Figure~\ref{cap:conf:list}i
\item $\hat{Y}$ belongs to $\HHH(1,1)$ if I is removed in Figure~\ref{cap:conf:list}j.
\item $\hat{Y}$ is a cylinder, if II is removed in Figure~\ref{cap:conf:list}j
\item $\hat{Y}$ belongs to $\HHH(2)$, if I is removed in Figure~\ref{cap:conf:list}k.
\item $\hat{Y}$ belongs to $\HHH(2)$, if II is removed in Figure~\ref{cap:conf:list}k.
\end{itemize}

By Lemma \ref{H(2)}, the vertical direction is completely periodic
when $\hat{Y} \in \HHH(2)$. We have to check the hypothesis of Lemma
\ref{H(1,1)} when $\hat{Y}$ belongs to $\HHH(1,1)$. We treat the case where
the cylinder I is removed in Figure~\ref{cap:conf:list}c. The analysis
in the other cases is more or less the same. The surface $Y$ contains
two vertical saddle connections isometric to the saddle connections
labelled by 1 and 2 in Figure~\ref{cap:conf:list}c. These connections don't disconnect
the surface thus, by Lemma~\ref{H(1,1)}, the flow in the vertical
direction is completely periodic on $\hat{Y}$. Therefore the vertical
direction is completely periodic on $(X,\omega)$.

Combining  Theorem \ref{periodic} and the previous argument, we
immediately deduce that there is a dense set of completely periodic
directions on $(X,\omega)$. The theorem is proven.
\end{proof}

\section{Detecting the closure: 
reduction steps} \label{reductionsteps}

In this section we show in several instances that a completely periodic direction 
on a surface $(X,\omega)$ plus some irrationality condition implies the
existence of a more useful completely periodic direction on a surface $(Y,\eta)$
in the $\GL^+_2(\R)$-orbit closure of $(X,\omega)$. Moreover,  $(Y,\eta)$
can be chosen such that the new completely periodic direction still 
satisfies some irrationality condition. The underlying strategy of all lemmas in this 
section is roughly the same. Finally, we arrive in a situation where we can show
that the $\GL^+_2(\R)$-orbit closure of $(Y,\eta)$ is all of $\LLL$. 

A bit more precisely, we start with a direction $p$ on $(X,\omega)$ for which there exists a 
splitting of the surface into some cylinders and/or tori. We apply
cyclic unipotent group $U$ (with eigenvector $p$) on
$(X,\omega)$. This action stabilizes globally the splitting and acts
on the product of tori and/or cylinders. We then apply Ratner's
theorem on orbit closures of the cyclic unipotent group $U$ inside
$U^k$ or $\Sl^k$ for $k=2,3$. The orbit closure is of the form $H\cdot
(X,\omega)$, where $H$ is a closed subgroup of $\Sl^k$ for $k=2,3$, in
the closure of the $\Sl$-orbit of $(X,\omega)$. Then with ``generic''
conditions, we get that $H$ is as big as possible use this to choose $(Y.\eta)$.

Two types of splitting directions will be particularly useful. The first one is a 
{\em $3C$-direction}; this is a completely periodic direction with
three cylinders and a configuration of saddle connections as the vertical direction
of Figure~\ref{cap:conf:list}d. Of course, this is not the only 
direction with three cylinders, but we stick to this short notation also used
in \cite{HLM}. The second one is defined as follows.

\begin{Definition}
A {\em $2T_\fix 2C$-direction} is a direction
with two simple cylinders exchanged by the hyperelliptic involution and
such that the complement decomposes into two tori. The vertical
direction in Figure~\ref{cap:conf:list}e is an example, which is in
fact completely periodic. A {\em $2T_\fix 2C$-direction} is not always
completely periodic. In fact if a {\em $2T_\fix 2C$-direction} is
not completely periodic, the flow, in that direction, contains one or 
two minimal component(s), corresponding to the two tori. We will say that a {\em
  $2T_\fix 2C$-direction} is  {\it irrational} if it is not completely
periodic.   
\end{Definition}

\begin{Remark}
We alert a reader who has also read \cite{HLM}, that there
such a direction was simply called $2T2C$. But we need to distinguish $2T_\fix 2C$-directions 
from {\em $2T2C_\fix$-directions}. A {\em completely periodic
$2T2C_\fix$-direction} (e.g. in Figure~\ref{cap:conf:list}f)
has a two simple cylinders {\it fixed} by the hyperelliptic involution and, 
cutting along all saddle connections in that direction, the complement 
consists also of two components, exchanged by the hyperelliptic involution.  
These components can be glued in a obvious way to two tori.
It will be convenient also to talk of {\em irrational 
$2T2C_\fix$-directions}, although this is some abuse of terminology.
Such a direction (cf Figure~\ref{cap:2T2Cfixuneq}a) also has two 
simple cylinders fixed by the 
hyperelliptic involution, but no saddle connections in the
complement. This complement is then connected and thus it is a
surface of genus two with two slits (rather than two tori).
\end{Remark}

We now give some useful definitions for the upcoming lemma.

\begin{Definition}
For any lattice $\Lambda$ we denote by $\Lambda'$ the normalized area
one lattice obtained by rescaling $\Lambda$ by a homothety.
We say that two lattices $\Lambda_1$ and $\Lambda_2$ are 
{\em strongly non-commensurable with respect to
a direction $p \in \R^2$}, if there is no unipotent element
$u$ with eigenvector $p$ such that $u(\Lambda'_1)$ and $\Lambda'_2$
(or $\Lambda'_1$ and $u(\Lambda'_2)$) are commensurable.
\end{Definition}

\begin{Definition}
We say that $3$ moduli $m_i$ of cylinders, in a completely periodic direction 
are {\em pairwise incommensurable}, if $m_i/m_j \neq \Q$ 
for $i \neq j$, i.~e.\ if they are $\Q$-linearly independent
or admit a $\Q$-linear relation $\sum_{i=1}^3 a_i m_i =0$ unique up to
scalars and with $a_i \neq 0$ for all $i$. 
     %
\end{Definition}

The first lemma consists of the key argument of~\cite{HLM}. We recall
the proof in details. From Lemma~\ref{2T2Cstronglyindep} to
Lemma~\ref{cpcase3fix3} we will relax the hypothesis on the translation
surface $(X,\omega)$.

\begin{Lemma} 
\label{3Cstronglyindep}
Let us assume that $(X,\omega) \in \LLL$ contains a
$3C$-direction. Then it contains also a transverse $2T_\fix
2C$-direction, denoted by $p$: it is given by the four homologous
saddle connection drawn in Figure~\ref{fig:surface-Ratner:3C}. Let us assume that the
direction $p$ has irrational and strongly non-commensurable tori. 
Then the $\GL^+_2(\R)$-orbit closure of $(X,\omega)$ is $\LLL$.
\end{Lemma}

\begin{proof}
Let us denote by $Z$ the closure of $(X,\omega)$ under $\Sl$
inside $\LLL_1$, the real hypersurface of translation surfaces of
$\omega$-volume one. One has to show that $Z = \LLL_1$. Note that
$\LLL_1$ has real dimension $11$.
Let $U$ the unipotent subgroup of $\Sl$ generated by unipotent elements 
$u$ having $p$ has eigenvector. Then the action of $U$ on $(X,\omega)$
is very simple: it stabilizes globally the direction $p$ and it acts
on each component of the splitting. Thus $U$ acts on the space of
a pair of tori and a cylinder, which is isomorphic to $\left(\Sl/\textrm{SL}_2(\Z)\right)^2
\times U$. Thanks to Ratner's theorem, the closure of $U\cdot
(X,\omega)$ is algebraic, i.~e.\ $H\cdot (X,\omega)$ where $H$ is a
closed subgroup of $\Sl^2\times U$ containing $U$ diagonally embedded.

The hypothesis that $p$ is irrational and strongly non-commensurable  
implies that this splitting is ``generic'' for
Ratner, i.~e.\ $H=\Sl^2\times U$. In other words, the closure
of the unipotent group orbit in that direction contains all surfaces
with the ratios of the splitting pieces fixed. Note that the $2T_\fix
2C$-direction is characterized by the four homologous saddle
connections. Thus we have proved that $Z$ contains a subset of real
dimension $(2*3+1)+2 = 9$ where $2*3$ stands for the dimension of
$\Sl^2$, $1$ for the dimension of $U$ and $2$ for the connection vector of the 
four homologous saddle connections realizing the splitting. We catch
the two missing dimensions as follows.

During the previous discussion, the ratios of the area of the tori
by the area of the cylinder are fixed. Thus one has to vary these two
ratios.

Let us consider the direction $\tilde{p}$ obtained by applying a simple
Dehn twist around the vertical non-simple cylinder of the
$3C$-direction (see Figure~\ref{fig:surface-Ratner:3C}). The direction $\tilde{p}$ is
obviously again of type $2T_\fix 2C$. Since the saddle connections of
a $2T_\fix 2C$-direction are homologous, the splitting in the
direction $\tilde{p}$ still exists in a neighborhood $N$ of
$(X,\omega)$. 

\begin{figure}[htbp]
\psfrag{c1}{$C_1$} \psfrag{c2}{$C_2$}
\psfrag{t1}{$T_1$}\psfrag{t2}{$T_2$}

\begin{center}
\includegraphics[width=2cm]{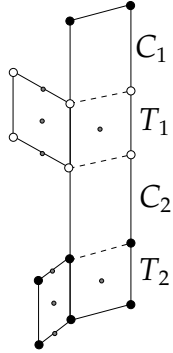}
\end{center}
\caption{
\label{fig:surface-Ratner:3C}
A ``diagonal'' resplitting $2T_\fix 2C$-direction on a $3C$-direction.
}
\end{figure}

We write $(X,\omega) = T_1\ \#\ C_1\ \#\ T_2\ \#\ C_2$ to denote
that $(X,\omega)$ is obtained as the connected sum of the $T_i$
and $C_i$ glued along a fixed set of slits, which is supressed
in the notation. We then know that orbit closure of $(X,\omega)$ 
ontains $U T_1\ \#\ C_1\ \#\
U T_2\ \#\ C_2$, where $U$ is the unipotent vertical element. In
particular for $(u_1,u_2)$ in a neighborhood 
of $(0,0)$ the twisted decomposition persists on
$(X(u_1,u_2),\omega(u_1,u_2)) := U_{u_1} T_1\ \#\ C_1\ \#\
U_{u_2}T_2\ \#\ C_2$. Let us denote the splitting in the twisted
direction by prime $(X'(u_1,u_2),\omega'(u_1,u_2))=T'_1\ \#\ C'_1\ \#\
T'_2\ \#\ C'_2$. One can show (see~\cite{HLM}, Lemma $5.9$) that
$(u_1,u_2)$-twisting can indeed be used to adjust the areas. More
precisely the map
$$
\varphi: 
(u_1,u_2) \mapsto ({\rm area}(T'_1(u_1,u_2))/{\rm area}(C'_1(u_1,u_2)),
{\rm area}(T'_2(u_1,u_2))/{\rm area}(C'_1(u_1,u_2)))
$$
is an invertible function in a neighborhood of $(u_1,u_2)=(0,0)$.
\par
Thus for almost all $(u_1,u_2)$ in a neighborhood of $(0,0)$ with
respect to the Lebesgue measure, we can apply Ratner's theorem in the new splitting
direction. This provides us a subset $N\subset Z$, of positive measure,
consisting of all above possible splitting with varying ratios. Recall
that the geodesic flow is ergodic on the hyperelliptic locus
$\LLL$ (\cite{Mas82,Ve1}). Therefore $Z$ has full measure in $\LLL$ and the lemma is
proven.
\end{proof}

\begin{Lemma} 
\label{2T2Cstronglyindep}
Let us assume that $(X,\omega)\in\LLL$ contains an irrational $2T_\fix
2C$-direction. In addition we assume that the two tori of the
splitting, are strongly non-commensurable. Then the
$\GL^+_2(\R)$-orbit closure of $(X,\omega)$ is $\LLL$.
\end{Lemma}

\begin{proof}
The idea is to apply previous Lemma~\ref{3Cstronglyindep}. Let $p$
denote the direction of the $2T_\fix 2C$-direction. As usual, 
let $U$ be the unipotent subgroup of $\Sl$ generated by unipotent elements 
$u$ having $p$ has eigenvector. By Ratner's theorem, in the
$\SL_2(\R)$ orbit closure of $(X,\omega)$, one finds surfaces $H\cdot
(X,\omega)$ where $H$ is a closed unimodular subgroup of $\Sl^2 \times
U$. By hypothesis, $H$ contains at least the product $\Sl^2$. Thus one
can apply a unipotent matrix in each torus independently,
in order to obtain a surface $(Y,\eta)$ in the $\Sl$-orbit closure of
$(X,\omega)$ with a $2T_\fix 2C$-direction and an adequate transverse
$3C$-direction (cf Figure~\ref{fig:surface-Ratner:3C}). 

Of course, by construction, the tori of the $2T_\fix 2C$-direction on
$(Y,\eta)$ are still irrational. By modifying slightly $(X,\omega)$ by
a diagonal matrix in $\Sl$ before applying Ratner's theorem, one can
obtain $(Y,\eta)$ without destroying above properties and in addition
two strongly non-commensurable tori. We then conclude by applying
Lemma~\ref{3Cstronglyindep}.
\end{proof}

\begin{Lemma} 
\label{2T2Cirrat}
If $(X,\omega)\in \LLL$ contains an irrational $2T_\fix
2C$-direction, then the $\GL^+_2(\R)$-orbit closure of $(X,\omega)$ is
$\LLL$.
\end{Lemma}

\begin{proof}
The strategy is to use previous Lemma~\ref{2T2Cstronglyindep}. We will
prove that there exists a surface $(Y,\eta)$ in the
$\GL^+_2(\R)$-orbit closure of $(X,\omega)$ with an irrational
$2T_\fix 2C$-direction containing two strongly non-commensurable
tori.

Up to a rotation, we can and will assume that the
$2T_\fix 2C$-direction is vertical as in
Figure~\ref{cap:2Tfix2Cirrat}.
\begin{figure}[htbp]
\label{fig:2Tfix2Cirrat}
\begin{center}
\epsfig{figure=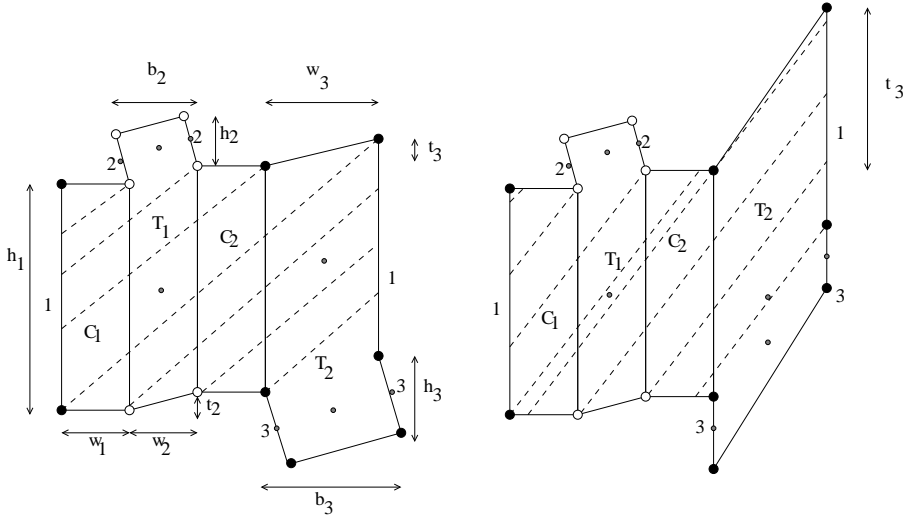,width=12cm} \qquad \qquad  
\caption{A $2T_\fix 2C$-direction with an irrational torus
and a ``diagonal'' resplitting with $p > t_3/w_3$ on the left figure
and $p < t_3/w_3$ on the right figure}
\label{cap:2Tfix2Cirrat}
\end{center}
\end{figure}
We apply the unipotent subgroup $\left( \begin{smallmatrix} 1 & 0 \\ * & 1
\end{smallmatrix}\right)$ to the vertical splitting. Exchanging
the role of $T_1$ and $T_2$ if necessary, one can assume that $T_1$ is
irrational. As in previous lemma, we denote the splitting $(X,\omega)
= C_1\ \#\ T_1\ \#\ C_2\ \#\ T_2$. As usual, using Ratner's theorem
one deduces that the $\Sl$-orbit closure of $(X,\omega)$ contains 
$C_1\ \#\ \Sl\cdot T_1\ \#\ C_2\ \#\ T_2$. Said differently,  the closure
contains all surfaces with the
following properties. All the parameters $w_1$, $w_3$, $h_1$, $h_3$ and $b_3$ 
are as given by the original surface $(X,\omega)$. Note that the torus
$T_2$ is not necessarily periodic. If it is, one has $b_3=w_3$, as
indicated in the right figure.

The parameters $w_2$, $t_2$ and $b_2$ can be chosen arbitrarily as
long as $b_2 > w_2$ and $h_2$ can be chosen to satisfy the volume
condition $(h_1+h_2)\cdot w_2 =A_1$ where $A_1$ is the volume of the
irrational torus $T_1$ of $(X,\omega)$. In particular $w_2$ can be
chosen close to $0$. By flipping the figure and changing the sign of
$t_2$, we may suppose that both $t_2$ and $t_3$ are non-negative. \medskip

Let us denote the splitting in the new direction (prescribed by the
dotted line) by prime. This new direction, by construction, crosses
$k$-times the top of each vertical cylinder $C_i$ as indicated on the
figure ($k=0$ in Figure~\ref{cap:2Tfix2Cirrat} on the left and $k=1$
on the right). This direction has slope
$$
p=\frac{(k+1)h_1+t_2+t_3}{2w_1+w_2+w_3}.
$$
Moreover one can move parameters $w_2$, $t_2$ and $b_2$ in order to
obtain a $2T_\fix 2C$-direction for the new splitting by the following
way. 
There are two cases to consider: $p > t_3/w_3$ (see the left figure)
and for $p<t_3/w_3$ (see the right figure). The case of equality
can be avoided by modifying $t_2$ slightly. Note that using the
hyperelliptic involution it is enough to check the intersection
behaviour for the dotted saddle connections emmanating from the white
singularity on the bottom of $C_2$.
\par
Choosing $w_2$ close to $0$ and $t_2/w_2$ small enough, the line
emanating at the white point has the correct intersection behaviour if
and only if
$$ 
\cfrac{kh_1+t_3}{w_1+w_3} < p < \cfrac{(k+1)h_1+t_3}{w_1+w_3}.
$$
This condition can always be satisfied choosing $k$ suitably
and equality can again be excluded by modifying $t_2$
slightly. Therefore the dotted line provides a new splitting $C'_1\
\#\ T'_1\ \#\ C'_2\ \#\ T'_2$.

Thus, until now, we have found a resplitting in a $2T_\fix 2C$-direction for
all $t_2, \ w_2$ and $h_2$ in a small open intervals.
We will check that for almost all triples $(t_2,w_2,b_2)$, subject to the condition
that $(h_1+h_2)w_2=A_1$ is fixed, this splitting
satisfies the conditions of Lemma~\ref{2T2Cstronglyindep}. Namely, for almost all triples the
two new tori $T'_1$ and $T'_2$ are irrational and strongly non-commensurable.

We give detailed proof for the left figure case; the right figure case
being completely similar. Thus the lattices of the tori (for the
resplitting $2T_\fix 2C$-direction) are 
$$
\Lambda'_1 = \left\langle 
v'_1:=\left(\begin{matrix}
w_2 \\ t_2
\end{matrix} \right),\ 
w'_1:=\left(\begin{matrix}
2w_2 - b_2 + 2w_1 + w_3 \\ (k+1)h_1+h_2+t_2 +t_3 
\end{matrix} \right)
\right\rangle
$$
and 
$$ 
\Lambda'_2 = \left\langle 
v'_2:=\left(\begin{matrix}
w_3 \\ t_3
\end{matrix} \right),\ 
w'_2:=\left(\begin{matrix}
2w_3-b_3 + 2w_1 + w_2 \\ (k+1)h_1+h_3+t_2+t_3 
\end{matrix} \right)
\right\rangle.
$$
One has to make sure that these two lattices do not possess vectors in
the direction of $p$. Moreover we want that $\Lambda_1$ and
$\Lambda_2$ are strongly non-commensurable with respect to $p$ in
order to apply previous lemma. A direct calculation shows that 
if $\Lambda_1$ (respectively $\Lambda_2$) possesses vectors in the
direction of $p$ then there exists $n\in \Z$ such that
$$
\cfrac{n t_2+h_2}{(k+1)h_1+h_3} \in \Q \qquad \left(
\textrm{respectively } \cfrac{n t_3+h_3}{(k+1)h_1+h_2} \in \Q \right).
$$
One gets similar conditions for the strong non-commensurability of the lattices. These two
conditions exclude only a $2$-dimensional set of Lebesque measure (in
$w_2,\ t_2,\ h_2$ and $b_2$ with $(h_1+h_2)w_2=A_1$) equals to
zero. Therefore we have found a resplitting in a $2T_\fix
2C$-direction, for almost all $t_2, \ w_2$ and $h_2$, with two
irrational and strongly non-commensurable tori. We then conclude the
proof of Lemma~\ref{2T2Cirrat} by using Lemma~\ref{2T2Cstronglyindep}.
\end{proof}

\begin{Lemma} 
\label{3Ccpisok}
If $(X,\omega)$ has a $3C$-direction with incommensurable moduli, the
$\SL_2(\R)$ orbit closure of $(X,\omega)$ is $\LLL$.
\end{Lemma}

\begin{proof}
If $p$ denotes the $3C$-direction, let $U$ be the subgroup of $\Sl$
generated by unipotent elements with eigenvector $p$. We will also
denote by $\tilde p$ the transversal $2T_\fix 2C$-direction (see
Figure~\ref{fig:surface-Ratner:3C}). We use Ratner's theorem for $U$
on the space of triples of normalized lattices $\left(
\Sl/\textrm{SL}_2(\Z)\right)^3$ to ensure that the direction $\tilde
p$ possesses an irrational torus. Then we conclude by using
Lemma~\ref{2T2Cirrat}.
\end{proof}

In order to prove similar lemmas for other completely periodic directions,
incommensurability is not strong enough. This will be clear from the
covering constructions in Section~\ref{existdirection}.
In Section~\ref{bouilla} we prove
that if $(X,\omega)$ is a pseudo-Anosov surface and if the
trace field is totally real and has degree $3$ over $\Q$ then any
non-parabolic completely periodic direction with $3$ moduli of
cylinders is pairwise incommensurable.
(Corollary~\ref{commensur}). Hence, in the following lemmas, we can replace
``pairwise incommensurable moduli'' by ``non-parabolic'' for surfaces arising
from Thurston's construction.

\begin{Lemma} 
\label{2T2Ccpisok}
If $(X,\omega)$ has a completely periodic $2T_\fix 2C$-direction with
pairwise in\-commen\-surable moduli, then the $\GL^+_2(\R)$-orbit closure
of $(X,\omega)$ is $\LLL$.
\end{Lemma}

\begin{proof}
Using the action of $\Sl$ we may suppose that the 
$2T_\fix 2C$-direction is vertical and that the bottom of
the cylinders $C_i$ is horizontal. 
Suppose that the saddle connections of slope 
$p:=(h_1+t_2+t_3)/(2w_1+w_2+w_3)$ intersect the vertical
saddle connections as drawn in 
Figure~\ref{cap:2Tfix2Cincomm}.
\begin{figure}[htbp]
\label{fig:2Tfix2Cincomm}
\begin{center}
\epsfig{figure=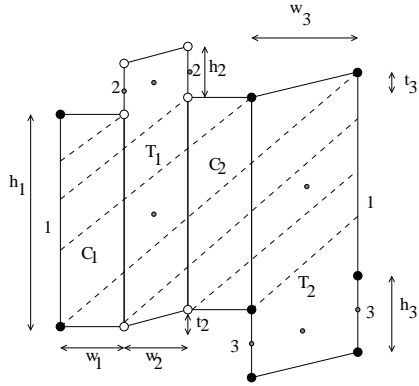,width=5.5cm} \qquad \qquad  
\caption{Resplitting of a $2T_\fix 2C$-direction}
\label{cap:2Tfix2Cincomm}
\end{center}
\end{figure}
Then the direction of slope $p$ is again $2T_\fix 2C$ and we want to
make sure that, say, the torus containing the top of $T_2$ is not periodic.
Sliding the Weierstrass points along the slope $p$ we observe that
this unwanted property holds if
$$ \frac{h_2+h_1}{2p} - \frac{2w_1+w_2+w_3}{2} \in \Q w_3.$$
Equivalently, we have to ensure that
$t_3 + t_2$ does not lie in a coset of $\Q$ in $\R$.
By hypothesis the orbit closure of $(X,\omega)$ under the
vertical unipotents contains a one-parameter subgroup that
fixes $C_1$ and hence $C_2$. 
Let $m_1 = w_1/(h_1+h_2)$ and $m_2 = w_2/(h_1+h_3)$ be the modulus of
$T_1$ and $T_2$ respectively. Simple cylinders persist
under small deformations. For $u \in \R$ small enough,
we may thus replace $t_1$ by $t_1+u$ and $t_2$ by $t_2 + um_2/m_1$
to obtain a surface in the unipotent orbit closure of
$(X,\omega)$ still with a  $2T_\fix 2C$-direction.
Since $m_2/m_1$ is irrational by pairwise incommensurabilty, 
we can choose $u$ in order to avoid the unwanted coset.
We can now apply Lemma~\ref{2T2Cirrat}.

It remains to check that one can always choose $t_2$ and
$t_3$ with the desired intersection property. This property
is surely satisfied for both $t_2$ and $t_3$ close to zero.
Recall that a twist $t_2$ (resp.\ $t_3$) is only well-defined up
to integer multiples $h_1+h_2$ (resp.\ $h_1+h_3$). It suffices
to find the resplitting not on $(X,\omega)$ but on a
surface in the orbit closure of vertical unipotents. Hence
it suffices to find $u \in \R$ such that
$$ 
uw_2 + t_2 \mod h_1+h_2 \quad \text{and} \quad
uw_3 + t_3 \mod h_1+h_3
$$
are both close to zero. Such $u$ exist since $m_1/m_2$ is
irrational by pairwise incommensurability.
\end{proof}

\begin{Lemma} \label{cpcase2T2Cfix}
If $(X,\omega)$ has a completely periodic $2T2C_\fix$-direction 
(as in Figure~\ref{cap:conf:list}f or
Figure~\ref{cap:2T2Cfixeq} (a) below) and such that the moduli $m_i$ 
of the cylinders $\{C_1,C_2,T_1\}$ are pairwise incommensurable 
then the $\GL^+_2(\R)$-orbit closure of $(X,\omega)$ is $\LLL$.
\end{Lemma}

\begin{proof}
We distinguish two cases. Suppose the $2T2C_\fix$-direction 
is vertical as in Figure~\ref{cap:2T2Cfixeq} (a).
\begin{figure}[htbp]
\label{fig:2T2Cfixeq}
\begin{center}
\psfrag{gg}{$h_2+t_2+t_3$}
\subfigure[]{\epsfig{figure=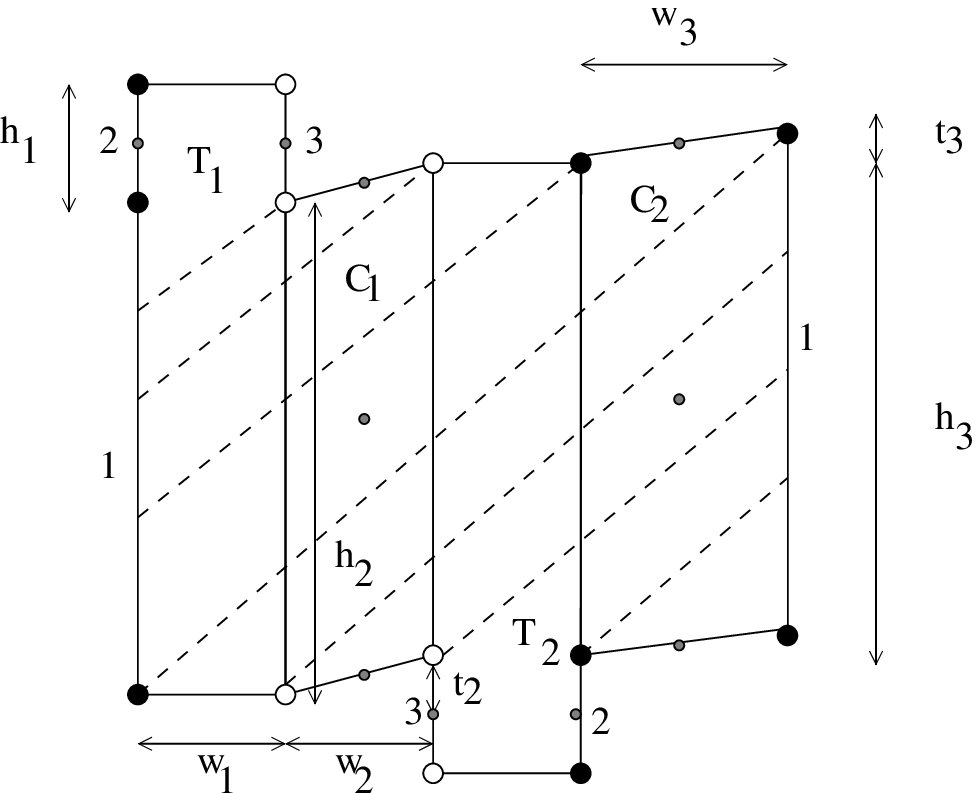,width=5.5cm}} \qquad \qquad  
\subfigure[]{\epsfig{figure=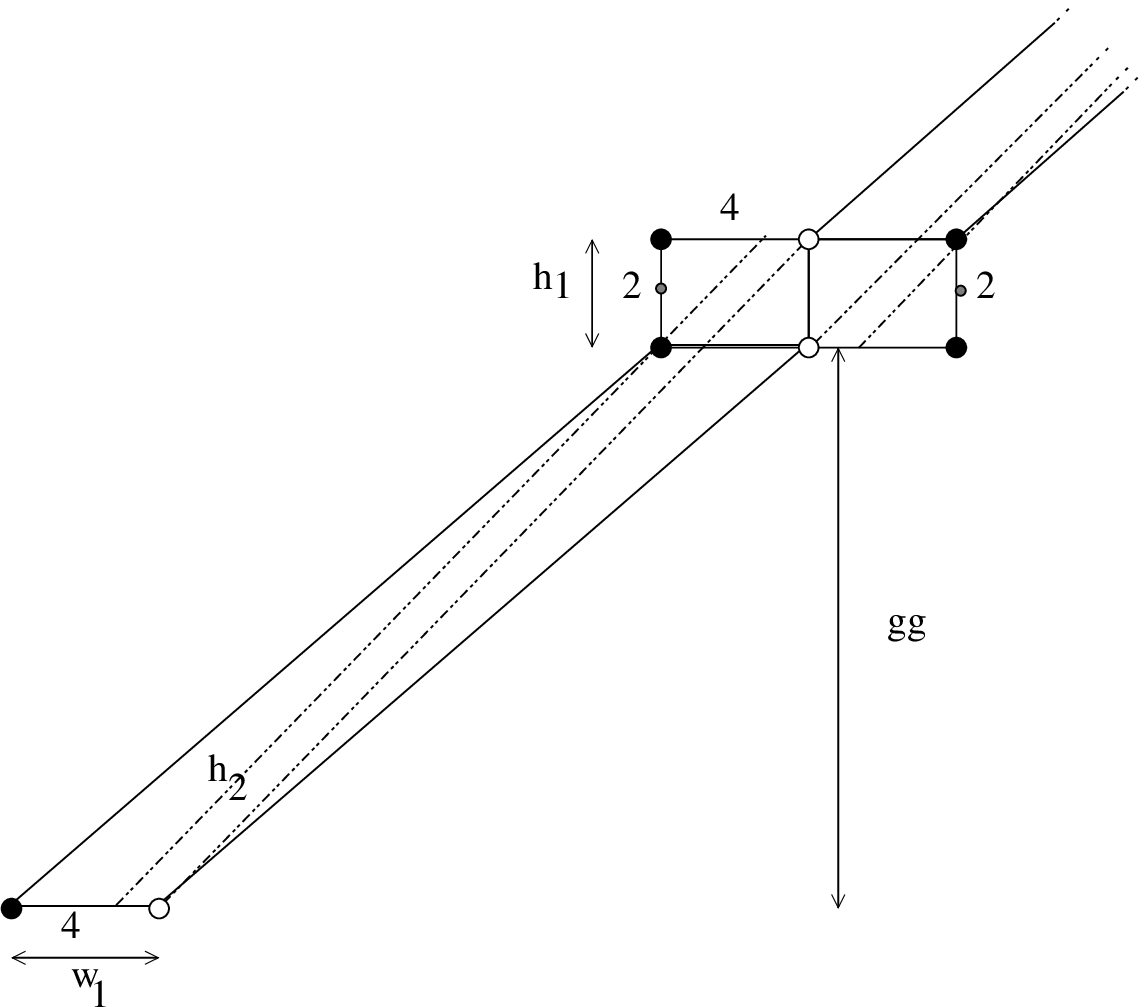,width=5.5cm}} 
\caption{Resplitting of a $2T2C_\fix$-direction
with $h_2=h_3$}
\label{cap:2T2Cfixeq}
\end{center}
\end{figure}
We first treat the case $h_2 = h_3$ as suggested by the figure.
This is quite analogous to the case of a $2T2C_\fix$-direction:
Using pairwise incommensurability we may suppose that $t_2$ and $t_3$
are both close to zero such that the indicated splitting exists.
Instead of resplitting in the direction of a 'simple Dehn twist
around the almost-cylinder of height $h_2$' we may also use
a $k$-fold Dehn twist, for $t_2$ and $t_3$ even closer (depending
on $k$) to zero. All these resplitting directions are irrational
 $2T2C_\fix$-directions and the complement of the two cylinders
looks as in Figure~\ref{cap:2T2Cfixeq} (b). The dotted lines
have slope
$$
\tilde{p} = \frac{h_2+t_2+t_3}{k(2w_1+w_2+w_3)-w_1}.
$$
If for some $m$
$$
\frac{h_1}{2mw_1} > \tilde{p} > \frac{h_1}{(2m+1)w_1},
$$
the direction of the dotted lines have $4$ homologous saddle
connections. Consequently $\tilde{p}$ is a a $2T_\fix 2C$-direction,
as illustrated in Figure~\ref{cap:2T2Cfixuneq} (a).

If we let $w:=2w_1+w_2+w_3$, the above condition for $\tilde{p}$
may be rephrased as
$$ \frac{k}{2m} \frac{wh_1}{w_1} > h_2 + t_2+t_3
+ \frac{h_1}{2mw_1} \quad \text{and} \quad
h_2 + t_2+t_3
+ \frac{h_1}{(2m+1)w_1} > \frac{k}{2m+1} \frac{wh_1}{w_1}.$$

If we choose $k/2m$ close to $h_2w_1/w h_1$ and $m$ large enough,
the inequalities are satisfied for $t_2$ and $t_3$ small enough.
This and the existence of the $k$-fold Dehn twist for
the chosen $k$ given two bounds for the size of $t_2$ and $t_3$.

If $t_2$ and $t_3$ chosen below both bounds and sufficiently
irrational, the direction $\tilde{p}$ is an irrational 
$2T_\fix 2C$-direction. Now Lemma~\ref{2T2Cirrat} applies.

We now treat the case that  $h_2 \neq h_3$, say $h_2 > h_3$. Consider the 
resplitting of Figure~\ref{cap:2T2Cfixuneq}.
\begin{figure}[htbp]
\label{fig:2T2Cfixuneq}
\begin{center}
\subfigure[]{\epsfig{figure=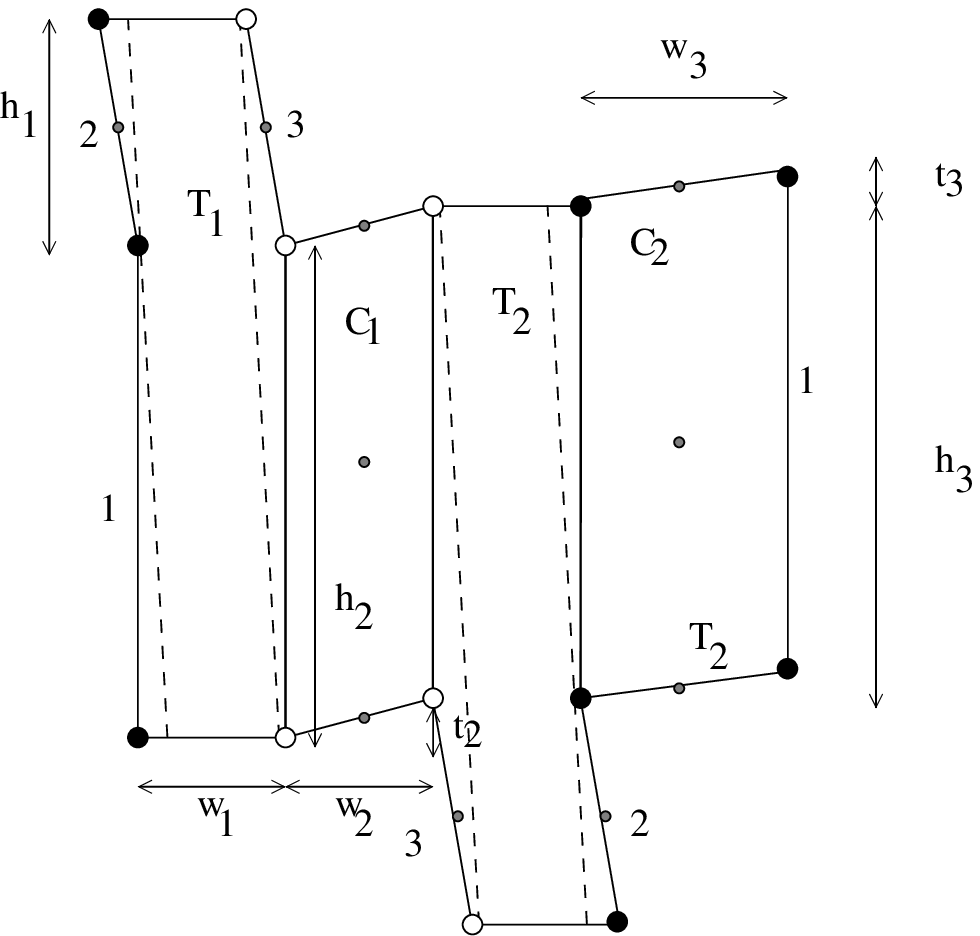,width=6cm}} \qquad \qquad  
\subfigure[]{\epsfig{figure=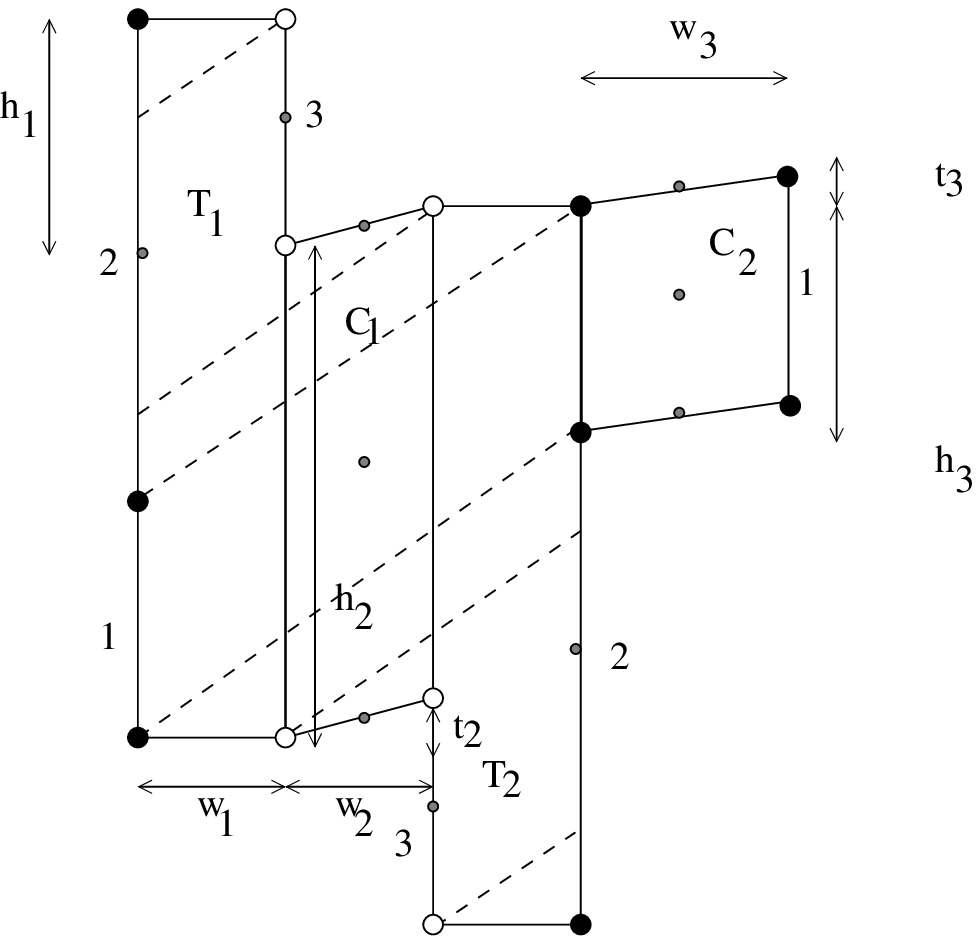,width=7cm}} 
\caption{(a) $2T_\fix 2C$-direction on an irrational
$2T2C_\fix$-direction and (b) resplitting of a $2T2C_\fix$-direction
with $h_2 \neq h_3$.}
\label{cap:2T2Cfixuneq}
\end{center}
\end{figure}
This is a $2T_\fix 2C$-direction $p$. Using incommensurability
and the unipotent action in the vertical direction we may
pass to an element in the unipotent orbit closure of $(X,\omega)$
and adjust $t_2$ to any suitable value, preserving the bottom
of $T_1$ horizontal.
The parameter $t_3$ will change, too, but this doesn't matter.
We can find $t_2$ such that the torus in the direction $p$ whose
intersection with $C_2$ is empty, is irrational. Lemma~\ref{2T2Cirrat}
applies.
\end{proof}

\begin{Lemma} \label{cpcase4fix1}
If $(X,\omega)$ has a completely periodic direction as in
Figure~\ref{cap:conf:list}g or Figure~\ref{cap:4Cylincomm} below
and such that the moduli $m_i$ of the cylinders $\{C_1,C_2,C_3\}$ are 
pairwise incommensurable then the $\GL^+_2(\R)$-orbit closure 
of $(X,\omega)$ is $\LLL$.
\end{Lemma}

\begin{proof}
We may suppose that the completely periodic direction is vertical.
The unipotent orbit closure of $(X,\omega)$ contains all surfaces as in 
Figure~\ref{cap:4Cylincomm}, where the bottom of $C_1$ is
horizontal and $t_2$ can be chosen arbitrarily.
\begin{figure}[htbp]
\label{fig:4Cylincomm}
\begin{center}
\epsfig{figure=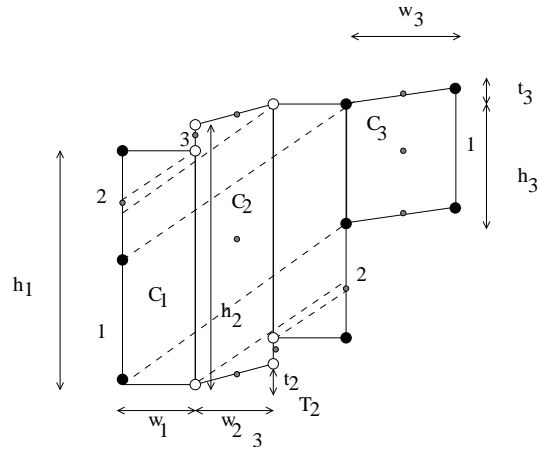,width=7cm} \qquad \qquad  
\caption{Resplitting of a direction
with $4$ cylinders only one of which is simple}
\label{cap:4Cylincomm}
\end{center}
\end{figure}
The resplitting give by the dotted lines is a $2T_\fix 2C$-direction $p$.
For suitable choice of $t_2$ the torus  in the direction $p$ that
does not intersect $C_3$ is irrational. We can now apply Lemma~\ref{2T2Cirrat}.
\end{proof}

\begin{Lemma} \label{cpcase3fix3}
If $(X,\omega)$ has a completely periodic direction as in
Figure~\ref{cap:conf:list}c or Figure~\ref{cap:3fix3incomm} below
and such that the moduli $m_i$ of the three cylinders are 
pairwise incommen\-surable, then the $\GL^+_2(\R)$-orbit closure 
$(X,\omega)$ is $\LLL$.
\end{Lemma}

\begin{proof}
Suppose the periodic direction is vertical as in 
Figure~\ref{cap:3fix3incomm}. 
\begin{figure}[htbp]
\label{fig:3fix3incomm}
\begin{center}
\epsfig{figure=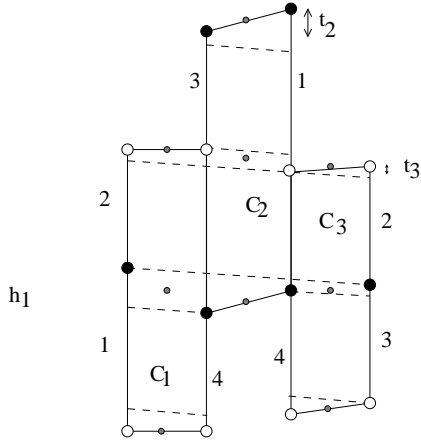,width=5.5cm} \qquad \qquad  
\caption{Resplitting of a direction with $3$ fixed cylinders}
\label{cap:3fix3incomm}
\end{center}
\end{figure}
Using a vertical unipotent, we may suppose that $C_1$ is 
untwisted, i.e.\ has a horizontal saddle connection.  For $t_2$ and $t_3$ 
sufficiently small but non-zero, the resplitting by dotted lines in 
Figure~\ref{cap:3fix3incomm} is a $2T_\fix 2C$-direction $p$. Moreover, 
for $t_2+t_3$ outside a coset of $\Q$ in $\R$ the direction 
$p$ has an irrational torus. We can arrange both conditions 
for a surfaces in the vertical unipotent orbit closure of $(X,\omega)$
by the same two arguments as in Lemma~\ref{2T2Ccpisok}.
\end{proof}

\section{Surfaces arising from 
Thurston's construction}
\label{bouilla}

In the whole section $(X,\omega)$ will be a surface with two transversal
parabolic directions belonging to the hyperelliptic locus $\LLL$. 
We prove that if $(X,\omega)$ is not a Veech surface for the most
obvious reason and if the pseudo-Anosov diffeomorphism is of a type that
can only arise for $g \geq 3$, 
then its $\GL^+_2(\R)$ orbit closure is large.

\begin{Theorem} 
\label{bouillaorbitclosure}
Let $(X,\omega)\in \LLL$ be a surface given by Thurston's
construction with trace field of degree $3$. By
Theorem~\ref{completely-periodic} this surface has infinitely many
completely periodic directions. Suppose that one of them is not
parabolic. Then
$$
\ol{\GL^+_2(\R)\cdot(X,\omega)} = \LLL.
$$
\end{Theorem}

We first study the properties of the moduli of cylinders for a non
parabolic completely periodic direction. We then prove the theorem at
the end of this section.

\subsection{Commensurability}

Let $(X, \omega)$ be surface given by Thurston's
construction. We recall that the
affine group $\SL(X,\omega)$ of $(X,\omega)$ contains a pseudo-Anosov 
diffeomorphism $\varphi$ with trace field $K$ and a parabolic element
$\rho$. Thanks to~\cite{HuLa} the field $K$ of $(X,\omega)$ is totally
real. Let us assume that $[K:\Q] = 3$.
Let $\sigma_i: K \to \R$ be the different real embeddings. We
fix one of them once and for all and for given $c\in K$
we shortly write $c':=\sigma_2(c)$ and $c'':=\sigma_3(c)$.
Choose $\tau_1, \tau_2 \in \Gal(\R/\Q)$ such that
$$\tau_1 \circ \sigma_1 = \sigma_2 \quad \text{and}
\quad \tau_2 \circ \sigma_1 = \sigma_3$$
and such that $\tau_1$ have order $3$ when restricted
to a Galois closure of $K/\Q$.

Suppose that the horizontal direction of $(X,\omega)$ is completely
periodic with cylinders of height $h_i$, circumference $c_i$ and moduli
$m_i := h_i/c_i$.

\begin{Lemma} 
\label{commensur}
Suppose the horizontal direction has $s=3$ or $s=4$ cylinders, such
that two of them, say the third and the fourth, are interchanged by
the hyperelliptic involution. Then either
\begin{itemize}
\item[i)] the direction is parabolic, i.e.\ $m_i/m_j \in \Q$ 
for all $(i,j)$, or
\item[ii)] the moduli are $\Q$-linearly independent, or 
\item[iii)] the moduli are related by 
$$
\sum_{i=1}^3 a_i m_i = 0, \quad \text{where} \quad a_i \in \Q^\ast,
$$
i.e.\ the moduli are not $\Q$-linearly independent but $m_i/m_j 
\not \in \Q$ for $i \neq j$.
In the terminology introduced in the previous section, the
direction is either parabolic or pairwise incommensurable.
\end{itemize}
\end{Lemma}

\begin{proof} 

Applying an upper triangular unipotent to $(X,\omega)$
we may suppose that the relative periods are in $K(i)$.
Let $t:=\tr \varphi$ and let 
$$\psi:= \varphi^*+(\varphi^*)^{-1}
\in {\rm End}(H^1(X,\R)).$$
By hypothesis on the trace field we have a decomposition into
eigenspaces of $\psi$ 
$$
H^1(X,\R) = S \oplus S' \oplus S''
$$
where we number the eigenspaces such that
 $S = \langle {\rm Re}(\omega), {\rm Im}(\omega)\rangle$
and $\tau_1(S) = S'$, $\tau_2(S)=S''$. Since $\psi$ 
is symplectic (see \cite{Mc1} Theorem~7.1 and \cite{Mc2}
Theorem~9.7), the eigenspaces are orthogonal with
respect to the cup product. The surface $(X,\omega)$
is covered by the cylinders and we conclude 

\begin{equation} 
\label{eq:flux1}
\sum_{i=1}^s m_i c_i c_i' = \sum_{i=1}^s h_i c_i' = 
\int_X {\rm Re}(\omega) \wedge {\rm Im} \tau_1(\omega) =
\frac{i}{4\pi}\int_X (\omega + \ol{\omega})  \wedge
(\tau_1(\omega) - \ol{\tau_1(\omega)}) = 0. 
\end{equation}
Replacing $\tau_1$ by $\tau_2$ we similarly have
\begin{equation} 
\label{eq:flux2}
\sum_{i=1}^s m_i c_i c_i'' = 0. 
\end{equation}
Moduli and circumferences of the cylinders exchanged
by the hyperelliptic involution are the same. 
We apply $\tau_2$ to this equation and subtract it from the
first one to obtain 
$$ \sum_{i=1}^3 (m_i - \tau_2(m_i)) \delta_i c_i c_i'' =0,$$
where $\delta_i = 1$, except for the case of $4$ cylinders
where $\delta_3 =2$.
Suppose that the lemma is wrong, i.~e.\ there is a relation
$a_1 m_1 + a_2 m_2 = 0$. 
Applying a matrix in $\SL_2(K)$ to $(X,\omega)$, we may suppose that
$m_3$ is rational without changing the ratios of the $m_i$. 
We deduce from the above equations
$m_i = \tau_2(m_i)$ for $i=1,2$. Hence in fact all $m_i \in \Q$
and we are in case i).
\end{proof}

\subsection{Completely periodic but not parabolic directions} 

The next lemma follows from Thurston's construction.

\begin{Lemma}
\label{lm:bound:cylinders}
Let $(X,\omega)$ be a surface given by Thurston's construction.
If $(X,\omega)$ admits a parabolic direction with a decomposition into $k$ cylinders, then
$[K:\Q] \leq k$. In particular, if $(X,\omega)$ admits a one cylinder
decomposition then it is an arithmetic surface.
\end{Lemma}

This lemma, combined with Lemma~\ref{commensur}, gives together

\begin{Lemma} 
\label{periodic-versus-degree}
Let $(X,\omega)\in \LLL$ be a surface given by Thurston's construction
with a completely periodic direction with two cylinders (or with three cylinders, two of
which interchanged by the hyperelliptic involution). Then
$[K:\Q] \leq 2$.
\end{Lemma}

\begin{proof}
We discuss the case of two cylinders, the other case is
similar since the exchanged cylinders have the same heights and widths.
The equations~\eqref{eq:flux1} and \eqref{eq:flux2} 
with $s=2$ hold in this situation, too.
They yield
$$ \frac{m_1}{m_2} = -\frac{c_2c_2''}{c_1c_1''}
= -\frac{c_2c_2'}{c_1c_1'}.$$
In particular $\tau_1(m_1/m_2)=m_1/m_2$. Therefore $m_1/m_2 \in \Q$
and the horizontal direction is a parabolic direction with $2$
cylinders. Applying Lemma~\ref{lm:bound:cylinders} we get $[K:\Q] \leq
2$, which ends the proof.
\end{proof}

In particular, we have:

\begin{Corollary}
\label{cor:direction}
Let $(X,\omega)\in \LLL$ be a surface given by Thurston's construction with
$[K:\Q]=3$. The configuration of every completely periodic direction
belongs to
Figure~\ref{cap:conf:list}c-\ref{cap:conf:list}d-\ref{cap:conf:list}e-\ref{cap:conf:list}f-\ref{cap:conf:list}g.
\end{Corollary}

\subsection{Proof of Theorem~\ref{bouillaorbitclosure}}
By Lemma~\ref{periodic-versus-degree} the completely periodic direction has at 
least $3$ orbits of cylinders under the hyperelliptic involution. 
By Theorem~\ref{cylclassH} and Corollary~\ref{cor:direction}, it is
thus one of the directions in Figure~\ref{cap:conf:list}c
or~\ref{cap:conf:list}d, Figure~\ref{cap:conf:list}e
or~\ref{cap:conf:list}f,
or Figure~\ref{cap:conf:list}g. 
The Lemmas~\ref{cpcase3fix3}, \ref{3Ccpisok},
\ref{2T2Ccpisok}, \ref{cpcase2T2Cfix}
and \ref{cpcase4fix1}, respectively, show that,
if the direction is not parabolic, the orbit closure
is as big as claimed. The hypothesis on pairwise incommensurability
is met because of Lemma~\ref{commensur}.

\section{Examples}
\label{sec:examples}

In this section we show that Theorem~\ref{bouillaorbitclosure} applies 
to infinitely many surfaces. 

\begin{Theorem}
There exist infinitely many surfaces given by Thurston's
construction, with trace field of degree $3$, and with 
a non parabolic periodic direction.
\end{Theorem}

We first construct such a surfaces (Lemma~\ref{lm:bouillabaisse}), and
we then prove that they possess a completely periodic {\em $2T_\fix
  2C$-direction}, which is not parabolic
(Proposition~\ref{prop:periodic:nonparabolic}).

\begin{Proposition}

Let $n$ be any positive integer and let $P=P_n$ be the polynomial 
$X^3-2(n^2+3)X^2+(7n^2+4)X-4n^2$. Then $P$ possesses only real roots.
If $\alpha=\alpha_n$ is the largest root, then $\alpha >1$. \medskip

\noindent Let us define the vectors 
$V=(V_1,V_2,V_3)$ and $H=(H_1,H_2,H_3)$ by
$$
\left\{ \begin{array}{l}
V_1 = \alpha-1 \\
V_2 = \frac1{n^2}(\alpha^2-\alpha(6+n^2)+4+n^2) \\
V_3 = 1
\end{array} \right. \textrm{ and }
\left\{ \begin{array}{l}
H_1 = 2V_1 \\
H_2 = n(V_1+V_2)\\
H_3 = V_1+V_3
\end{array} \right.
$$
then $V$ and $H$ are positive, namely $V_i >0$ and 
$H_i >0$ for $i=1,2,3$.
\end{Proposition}

\begin{proof}
If $A$ is the symmetric matrix defined by 
$A=\left( \begin{smallmatrix} 5+n^2 & n^2 & 1 \\
n^2 & n^2 & 0 \\ 1 & 0 & 1 \end{smallmatrix} \right)$, then 
the characteristic polynomial $\chi_A$ is $P$. This ensures that $P$ has
only reals roots. In addition 
$A$ is an irreducible matrix, indeed $(A^2)_{ij}>0$ for all $i,j$. 
This proves that $\alpha$ is the Perron--Frobenius eigenvalue of the entire matrix $A$ and
therefore $\alpha >1$. Let us show that $V$ is an eigenvector for 
the eigenvalue value $\alpha$ of the matrix $A$. The Perron--Frobenius
theorem will then show the positivity of $V$.

For that we have to show that $V$ is a solution of the linear system
\begin{equation}
\label{eq:linear}
\left\{
\begin{array}{lllll}
(5+n^2)V_1 & + n^2V_2 & +V_3 & = \alpha V_1 & (L_1)\\
n^2 V_1 & + n^2 V_2 & & = \alpha V_2 & (L_2)\\
V_1 & & + V_3 & = \alpha V_3 & (L_3)
\end{array} \right.
\end{equation}
This is a simple verification, according to the fact that 
$\alpha^3 = 2(n^2+3)\alpha^2-(7n^2+4)\alpha+4n^2$. The positivity 
of $H$ is then clear.
\end{proof}

Thanks to the previous statement, for any $n\geq 1$, 
let $(X_n,\omega_n)$ be the surface presented in 
Figure~\ref{cap:surface-thurston}. 

\begin{figure}[htbp]
\label{fig:surface-thurston}
\begin{center}
\psfrag{v1}{$\scriptstyle V_1$} \psfrag{v2}{$\scriptstyle V_2$}
\psfrag{v3}{$\scriptstyle V_3$} \psfrag{h1}{$\scriptstyle H_1$}
\psfrag{h2}{$\scriptstyle H_2$} \psfrag{h3}{$\scriptstyle H_3$}
\psfrag{nh2}{$\scriptstyle nH_2$}

 \includegraphics[width=8cm]{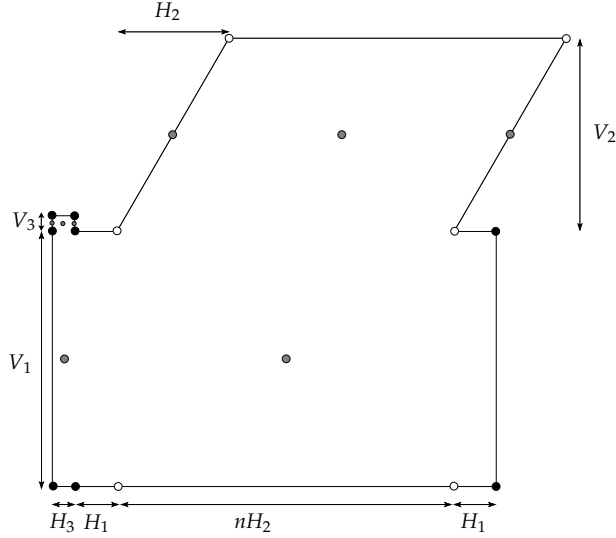}
\end{center}
\caption{
A surface $(X_n,\omega_n)$ in the hyperelliptic locus $\mathcal L$.
The vertical direction has been stretched. 
}\label{cap:surface-thurston}
\end{figure}

\begin{Lemma}
\label{lm:bouillabaisse}
The surface $(X_n,\omega_n)$ is a bouillabaisse surface. Moreover its
trace field is $\Q(\alpha_n)$ and its trace field has degree $3$ over $\Q$ if
$n > 2$.
\end{Lemma}

\begin{proof}

According to the horizontal direction, the surface is decomposed into
a completely periodic $3C$-direction (see
Figure~\ref{cap:surface-thurston}). Let us calculate  
the moduli of theses cylinders and show that they are equal. Using
$(L_1)$ and $(L_2)$ in the system~(\ref{eq:linear}), one checks that the moduli are
$$
\cfrac{H_3}{V_3}=V_1+V_3=\alpha,\quad \cfrac{H_3+2H_1+nH_2}{V_1}=\alpha,
\quad \text{and} \quad
\cfrac{nH_2}{V_2}=n^2\cfrac{V_1+V_2}{V_2}=\alpha.
$$
All the moduli are equal, thus $(X_n,\omega_n)$ is stabilized by the
parabolic element $T_{\alpha}=\left( \begin{smallmatrix} 1 & \alpha \\ 0 & 1
\end{smallmatrix} \right)$.

Now let us consider the vertical direction; this is a completely
  periodic {\em $2T_\fix2C$-direction}. The two simple cylinders have
  moduli $V_1/H_1=1/2$. The two others cylinders have moduli
$(V_1+V_3)/H_3=1$ and $n(V_1+V_2)(H_2)=1$. Thus the parabolic element 
$U_{1/2}=\left( \begin{smallmatrix} 1 & 0 \\ \frac1{2} & 1 \end{smallmatrix} 
\right)$ stabilises $(X_n,\omega_n)$.

Therefore $(X_n,\omega_n)$ is a bouillabaisse surface, since it has 
two transverse parabolic directions. Moreover 
$(X_n,\omega_n)$ is stabilised by the hyperbolic element $TU$ which
has the trace $\frac1{2}(4+\alpha)$. The trace field of
$(X_n,\omega_n)$ is then $\Q(\alpha_n)$. The next lemma shows that
$P_n$ is irreducible over $\Q$ if $n \geq 3$ which will end the proof.
\end{proof}

\begin{Lemma} \label{irreducibility}
Let $n$ be a positive integer.
The polynomial $P_n(X) = X^3-2(n^2+3)X^2+(7n^2+4)X-4n^2$ is
irreducible over $\Q$ if $n > 2$.
\end{Lemma}

\begin{proof}
Since $P_n$ is monic, 
it is irreducible over $\Q$ if and only if it is
irreducible over $\Z$. As it is of degree 3, it is irreducible over
$\Z$ if and only if it has no root in $\Z$. We check the following
equalities:

$$\left\{\begin{array}{lcl}
P_n(0) & = & -4n^2 \\
P_n(1) & = & -1+n^2 \\
P_n(2) & = & -8 +2n^2 \\
P_n(3) & = & -n^2 - 15 \\
P_n(2n^2+2) & = & -8 -14n^2- 2n^4 \\
P_n(2n^2+3) & = & -15-11n^2+2n^4,
\end{array}\right. \quad \text{thus for $n \geq 2$} \quad
%
\left\{\begin{array}{lcl}
P_n(0) & < & 0 \\
P_n(1) & > & 0 \\
P_n(2) & > & 0 \\
P_n(3) & < & 0 \\
P_n(2n^2+2) & < & 0 \\
P_n(2n^2+3) & > & 0.
\end{array}\right.
$$
By the intermediate value theorem, if $n > 2$, $P_n$ has 3 real roots
$\alpha$, $\beta$, $\gamma$ satisfying:
\begin{equation}
\label{eq:encadrement:racines}
\left\{\begin{array}{l}
0  <  \gamma < 1 \\
2  < \beta < 3 \\
2n^2+2 < \alpha < 2n^2+3 
\end{array}\right.
\end{equation}
Consequently, $P_n$ has no root in $\Z$ thus it is irreducible over $\Z$.
\end{proof}

\begin{Remark}
The polynomials $P_1$ and $P_2$ are not irreducible; indeed, 
if $n=1,2$ we have $P_1 = (X-1)(X^2-7X+4)$ and $P_2 = (X-2)(X^2-12X+8)$.
More precisely, $(X_1,\omega_1)$ is an unramified cover of a 
Veech surface belonging to $\mathcal H(2)$.
\end{Remark}

\subsection{Another completely periodic 
(but non-parabolic) direction on $(X_n,\omega_n)$}

\begin{Proposition}
\label{prop:periodic:nonparabolic}
Let $\Theta$ be the slope $V_2/H_2=n/\alpha$. Then the direction
$\Theta$ on $(X_n,\omega_n)$ is completely periodic {\em
  $2T_\fix2C$-direction} on $(X_n,\omega_n)$. Moreover, if $n>1$,
this direction is not parabolic.
\end{Proposition}

\begin{figure}[htbp]

\psfrag{g1}{$\gamma_1$} \psfrag{g2}{$\gamma_2$}
\psfrag{g3}{$\gamma_3$} \psfrag{g4}{$\gamma_4$}
\psfrag{g5}{$\gamma_5$} \psfrag{g6}{$\gamma_6$}

\begin{center}
 \includegraphics[width=7cm]{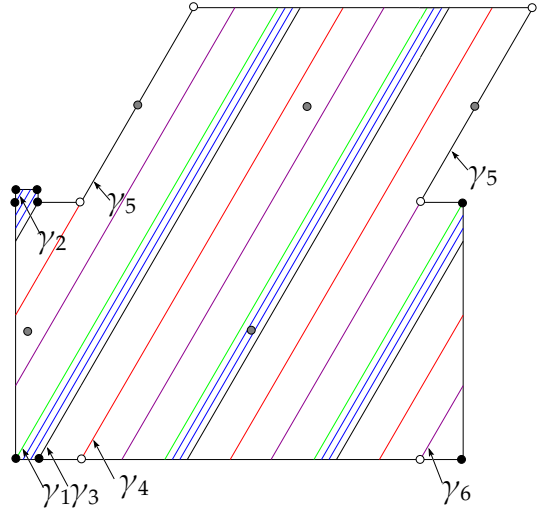}
\end{center}
\caption{
\label{fig:surface-thurston:nonparabolic}
Another completely periodic (non parabolic) {\em
  $2T_\fix2C$-direction} on $(X_n,\omega_n)$.
}
\end{figure}

We decompose the proof into three steps. We first prove that the direction 
$\Theta$ is completely periodic. We then compute the moduli of the
cylinders and finally we prove that one ratio is not rational.

\begin{Lemma}
The direction $\Theta$ is a completely periodic {\em
$2T_\fix2C$-direction}.
\end{Lemma}

\begin{proof}
By definition, $\gamma_5$ is closed and the corresponding holonomy vector is
$$
\int_{\gamma_5}\omega_n = \left( 
\begin{array}{c} H_2 \\ V_2 \end{array} \right).
$$
Let us prove that $\gamma_4$ is also closed.
Let $x_i$ be the $x$-coordinates of intersection point of $\gamma_4$
with the top (and bottom) of the horizontal cylinder of height
$V_1$ as illustrated in Figure~\ref{cap:surface-thurston:periodic}.
We will show that $x_i < H_1 + H_3 + n H_2$ for all 
$i=1,\dots,n-1$ and $x_n =3H_1 + 2H_3 + n H_2$ which will ensure that $\gamma_4$ 
is closed. 

\begin{figure}[htbp]
\label{fig:surface-thurston:periodic}
\begin{center}
\psfrag{x1}{$\scriptstyle x_1$}
\psfrag{x2}{$\scriptstyle x_2$}
\psfrag{x0}{$\scriptstyle x_0$}
\psfrag{x3}{$\scriptstyle x_3$} 
 \includegraphics[width=6cm]{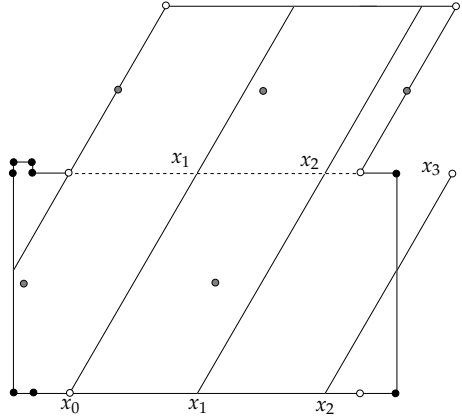}
\end{center}
\caption{
A periodic direction on $(X_3,\omega_3)$ in the direction $\Theta$.
}\label{cap:surface-thurston:periodic}
\end{figure}

Elementary geometry gives 
$\cfrac{V_1}{x_i-x_{i-1}} = \Theta = \cfrac{n}{\alpha}$, which yields  
\begin{equation}
\label{equation:xi}
x_i = \cfrac{\alpha V_1}{n} + x_{i-1} = \cfrac{i\cdot \alpha V_1}{n} + x_0 =
\cfrac{i\cdot \alpha V_1}{n} + H_1 +H_3.
\end{equation}
Let $\Delta_i = x_i - (H_1 + H_3 + n H_2)$. Recall that $nH_2=\alpha V_2$ and 
$V_1 = \left(\cfrac{\alpha}{n^2} -1 \right) V_2$ (thanks to $(L_2)$ in
Equation~(\ref{eq:linear})). Then 
$$
\Delta_i= \cfrac{i\cdot \alpha V_1}{n} - \alpha V_2 = 
\cfrac{\alpha V_2}{n}\left (i\left( \cfrac{\alpha}{n^2} -1 \right) -n \right).
$$

\begin{Claim}

One has
$$
\cfrac{\alpha}{n^2} -1 < \cfrac{n}{n-1}
$$
\end{Claim}

\begin{proof}[Proof of the Claim]
If $n>2$, we already proved that $\alpha < 2n^2+3$ (see
Equation~(\ref{eq:encadrement:racines})). Hence the claim 
follows, once we have shown $\cfrac{2n^2+3}{n^2} -1 = 1 +
\cfrac{3}{n^2} < \cfrac{n}{n-1}$. This is obvious if $n >2$.
The case $n=2$ is checked directly.
\end{proof}

Therefore, if $i \leq n-1$ we have 
$$
\Delta_i \ < \ \cfrac{\alpha V_2}{n}\left ((n-1)\cdot \cfrac{n}{n-1} - n \right) = 0.
$$
To complete the proof, one has to show that $x_n = 3H_1+2H_3+n H_2$. 
Equation~(\ref{equation:xi}) with $i=n$ gives $x_n = \alpha V_1 + H_1 + H_3$. Therefore
\begin{multline*}
x_n - (3H_1+2H_3+n H_2) =\alpha V_1 - 4 V_1 - V_1 - V_3 -n^2(V_1+V_2) = \\
= (5+n^2)V_1 + n^2 V_2 + V_3 - \alpha V_1 = 0 \qquad \textrm{ by } (L_1).
\end{multline*}
Hence $x_n = 3H_1+2H_3+n H_2$ and $\gamma_4$ is closed. In 
addition, we get the holonomy vector associated to $\gamma_4$:
$$
\int_{\gamma_4}\omega_n = \left( 
\begin{array}{c} * \\ (n-1)(V_1+V_2)+V_1 \end{array} \right).
$$
The same calculation shows that $\gamma_1,\gamma_3$ and $\gamma_6$ are 
also closed. It is not difficult to check what is the complement of
these four saddle connections: it gives two tori. Hence the surface
$(X_n,\omega_n)$ is decomposed into a {\em $2T_{fix}2C$-direction} in
the direction $\Theta$ (see Figure~\ref{cap:surface-decomposed}). The
torus $T_1$ is actually periodic. By Theorem~\ref{theo:CaSm} the SAF-invariant
vanishes in this direction and the SAF-invariant is equals to zero on cylinders
$C$ and torus $T_1$. Therefore the SAF-invariant is also equal to zero on $T_2$
and finally $\Theta$ is completely periodic. The lemma is proven.
\end{proof}

\begin{figure}[htbp]
\begin{center}

\psfrag{g1}{$\gamma_1$} \psfrag{g2}{$\gamma_2$}
\psfrag{g3}{$\gamma_3$} \psfrag{g4}{$\gamma_4$}
\psfrag{g5}{$\gamma_5$} \psfrag{g6}{$\gamma_6$}
\psfrag{h1}{$H_1$} \psfrag{h3sn}{$H_3/n$}
\psfrag{C}{$C$} \psfrag{T1}{$T_1$} \psfrag{T2}{$T_2$}

 \includegraphics[width=5cm]{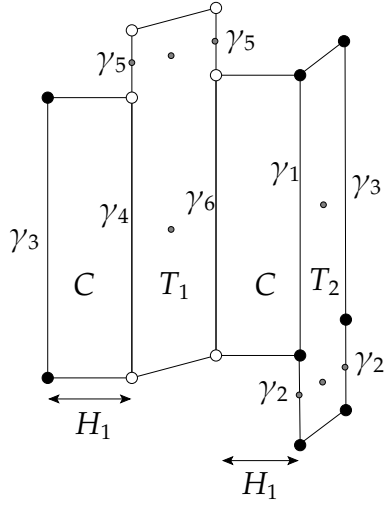}
\end{center}
\caption{
The decomposition of $(X_n,\omega_n)$ in the direction $\Theta$: this 
is a completely periodic {\em $2T_{fix}2C$-direction}.
}\label{cap:surface-decomposed}
\end{figure}

\begin{Lemma}
If $C,T_1,T_2$ denote the cylinder decomposition in the direction 
$\Theta$, then the ratio of the moduli of $C$ and $T_1$ is 
$$
r_1= \cfrac{m(C)}{m(T_1)} = \cfrac{(\alpha-n)(2n^3-n^2-(n-1)\alpha)}{2n(\alpha -n^2)}.
$$
\end{Lemma}

\begin{proof}
According to Figure~\ref{cap:surface-decomposed}, we can compute 
the circumference and the height of the cylinders $C$ and $T_1$.
The circumference of $C$ is just $|\gamma_4|$ where  
$$
|\gamma_4|^2 = ((n-1)(V_1+V_2)+V_1)^2\left( 1+\frac1{\Theta^2} \right),
$$
and the height is $H_1$.

The circumference of $T_1$ is just $|\gamma_4|+|\gamma_5|$ where  
$$
|\gamma_5|^2 = V^2_2\left( 1+\frac1{\Theta^2} \right),
$$
and the height is $h(T_1) = (H_1+H_3+nH_2)-x_{n-1}$ (see
Figure~\ref{fig:surface-thurston:nonparabolic}).
This can be simplified to
$$
h(T_1)= nH_2 - \cfrac{(n-1)\alpha V_1}{n} = \alpha V_2 - 
\cfrac{(n-1)\alpha}{n} \left( \cfrac{\alpha}{n^2}-1 \right) V_2
= \cfrac{\alpha V_2}{n^3}(2n^3-n^2-(n-1)\alpha).
$$
Taking the ratio we get:
$$
r_1 = \cfrac{\cfrac{|\gamma_4|}{H_1}}{\cfrac{|\gamma_4|+|\gamma_5|}{h(T_1)}} = 
\cfrac{|\gamma_4|}{|\gamma_4|+|\gamma_5|} \cdot \cfrac{h(T_1)}{H_1}.
$$
The quantities on the right hand side can be simplified to
$$
\cfrac{|\gamma_4|}{|\gamma_4|+|\gamma_5|} = 
\cfrac{(n-1)(V_1+V_2)+V_1}{(n-1)(V_1+V_2)+V_1+V_2} = 
\cfrac{\frac1{n}\alpha V_2 - V_2}{\frac1{n}\alpha V_2} = \cfrac{\alpha-n}{\alpha}
$$
and 
\begin{equation*}
\cfrac{h(T_1)}{H_1} = \cfrac{\frac1{n^3}\alpha V_2 \left(2n^3-n^2-(n-1)\alpha\right)}{2V_1} = 
 \cfrac{\alpha \left(2n^3-n^2-(n-1)\alpha\right)}{2n(\alpha -n^2)}.
\end{equation*}
Plugging this
into the previous equation, this gives
\begin{equation}\label{r1}
r_1 = \cfrac{(\alpha-n)(2n^3-n^2-(n-1)\alpha)}{2n(\alpha -n^2)}
\end{equation}
which is the desired equality.

\end{proof}

\begin{Lemma}
The following are equivalent:
\begin{enumerate}
\item $r_1 \in \Q$.
\item $r_1=\frac1{2}$.
\item $n=1$.
\end{enumerate}
\end{Lemma}

\begin{proof}
If $r_1 \in \Q$, using equation \eqref{r1}, we get in the basis
$\left\{1,\alpha,\alpha^2\right\}$:
$$
\left\{ \begin{array}{ll}
n^3(2n-1-2r_1) & = 0 \\
n(2n^2-1-2r_1)   & = 0 \\
n-1      & = 0
\end{array} \right.
$$
The solution gives $r_1=\frac1{2}$ and $n=1$. The lemma is proven.
\end{proof}

\section{Surfaces with a 
$2T_\fix2C$-direction and covering 
constructions} \label{existdirection}

The locus ${\mathcal D} \subset \LLL$ of translation surfaces having
a completely periodic direction of type $2T_\fix 2C$ has proven to be 
useful above.
Contrary to genus two, where pairs of saddle connections exchanged by the hyperelliptic involution {\it always} exist (\cite{Mc3}), ${\mathcal D}$ is not equal
to $\LLL$.

In order to show this, we consider coverings
$\pi: X \to Y$ of a flat surface $(Y,\eta)$ in $\HHH(2)$. 
Such a covering is unramified and has degree two.
Given $Y$, the covering $\pi$ is determined by the choice
of $\zeta \in H^1(Y,\Z/2)$ or of a line bundle $\MMM$ with
$\MMM^2 = \OOO_Y$. Recall (e.g.~\cite{KZ}) that
on $Y$ we have a natural spin structure (i.e.\ a square root
of the canonical bundle) given by $\OOO_Y(P)$,
if $Z(\eta) = 2P$. Since $P$ is a Weierstrass point, 
this spin structure has odd parity. $\OOO_Y(P) \otimes \MMM$
defines another spin structure on $Y$. By \cite{At71}, the space ${\mathcal C}$
of coverings $\pi$ has two components, distinguished by the
parity  of this spin structure $h^0(Y,\OOO_Y(P) \otimes \MMM) \mod 2$.
Let $Z(\omega) = 2P_1+2P_2$ be the preimage of $Z(\eta)$.
Since 
$$ H^0(X,\OOO_X(P_1+P_2)) \cong H^0(Y,\OOO_Y(P) \otimes \MMM) 
+ H^0(Y,\OOO_Y(P)),$$
this parity is given by the usual parity of the spin structure on $X$. 
Consequently, precisely one of the components of ${\mathcal C}$ lies in 
$\Omega M_3(2,2)^\odd $ and in fact automatically in the hyperelliptic
locus $\LLL$. We denote this component by ${\mathcal C}^\odd$.

\begin{Theorem}
The locus ${\mathcal D}$ is strictly contained in $\LLL$. 
More precisely, the intersection 
${\mathcal D}^c \cap {\mathcal C}^\odd$
is non-empty and strictly contained in  ${\mathcal C}^\odd$.
It consists of orbits of Veech surfaces.
\end{Theorem}

\begin{proof}
First we show that the intersection is not empty. Consider
a square-tiled covering of $6$ squares as Figure~\ref{cap:origami}.

\begin{figure}[htbp]
\begin{center}
 \includegraphics[width=5cm]{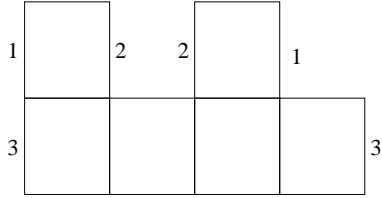}
\caption{A square-tiled surface in $\LLL$ without $2T_\fix 2C$-direction}
\label{cap:origami}
\end{center}
\end{figure}
Horizontal sides are glued by vertical translations.
Using e.g.\ \cite{Sc04}, one checks that the affine group of this
square-tiled covering has three cusps, given by the horizontal,
and vertical direction and by the direction of slope $1$.
None of these directions is a $2T_\fix 2C$-direction.

By \cite{C}, a surface in $\HHH(2)$ which is not Veech has a direction where it splits into a 
cylinder and a non-periodic torus. 
Consequently, a covering of a generic  surface in $\HHH(2)$
yields a surface in $\LLL$ with an irrational 
$2T2C_\fix$-direction. Such a surface also has a
$2T_\fix 2C$-direction by Figure~\ref{cap:2T2Cfixeq} (b).

Since the saddle connections involved in a $2T_\fix 2C$-direction
are homologous, ${\mathcal D}$ is open. Consequently,
${\mathcal D}^c \cap {\mathcal C}^\odd$ is closed and
$\Sl$-invariant. The last claim now follows from \cite{Mc3}. 
\end{proof}


\end{document}